\documentclass[reqno,11pt]{amsart}
\usepackage{a4wide,color,eucal,enumerate,mathrsfs}
\usepackage[normalem]{ulem}
\usepackage{amsmath,amssymb,epsfig,amsthm} 
\usepackage[latin1]{inputenc}
\usepackage{psfrag}
\usepackage{pdfpages}

\def\ez{\epsilon}

\def\bz{\beta}

\def\gz{{\gamma}}

\def\bint{{\ifinner\rlap{\bf\kern.35em--}
		\int\else\rlap{\bf\kern.45em--}\int\fi}\ignorespaces}

\def\bbint{{\ifinner\rlap{\bf\kern.35em--}
		\hspace{0.078cm}\int\else\rlap{\bf\kern.45em--}\int\fi}\ignorespaces}

\def\diam{{\mathop\mathrm{\,diam\,}}}

\newtheorem{thm}{Theorem}[section]
\newtheorem{lem}[thm]{Lemma}
\newtheorem{prop}[thm]{Proposition}
\newtheorem{cor}[thm]{Corollary}
\newtheorem{defi}[thm]{Definition}
\numberwithin{equation}{section}

\theoremstyle{remark}
\newtheorem{rem}[thm]{Remark}

\def\bint{{\ifinner\rlap{\bf\kern.35em--}
		\int\else\rlap{\bf\kern.45em--}\int\fi}\ignorespaces}

\title[Carrot John domains in variational problems]{Carrot John domains in variational problems}
\author{Weicong Su, Yi Ru-Ya Zhang}
\date{\today}
\address{Academy of Mathematics and Systems Science, the Chinese Academy of Sciences, Beijing 100190, P. R. China}
\email{suweicong@amss.ac.cn}  
\email{yzhang@amss.ac.cn}

\thanks{Both of the authors are funded by National Key R\&D Program of China (Grant No. 2021YFA1003100), the Chinese Academy of Science and NSFC grant No. 12288201.
}
\begin{document}

	\begin{abstract}
		In this paper, we explore carrot John domains within variational problems, dividing our examination into two distinct sections. The initial part is dedicated to establishing the lower semicontinuity of the (optimal) John constant with respect to Hausdorff convergence for bounded John domains. This result holds promising implications for both shape optimization problems and Teichm\"uller theory.
		
		In the subsequent section, we demonstrate that an unbounded open set satisfying the carrot John condition with a center at $\infty$, appearing in the Mumford-Shah problem, can be covered by a uniformly finite number of unbounded John domains (defined conventionally through cigars). These domains, in particular, support   Sobolev-Poincar\'e inequalities.
	\end{abstract}

	\maketitle
	
	\section{Introduction}
	In the realm of shape optimization problems, instances frequently arise wherein the objective is to identify the optimal class of sets, denoted as 
	$U$, based on the ratio of functionals that incorporate the norm of a specific class of Sobolev functions 
	$u$, the norm of its gradient $Du$, and the norm of its trace $u|_{\partial U}$ on 
	$\partial U$.  
	A prototypical illustration of such a scenario is the pursuit of the optimal sets $U\subset \mathbb R^n$ for the first 
	$p$-Dirichlet eigenvalue, where for $1<p<\infty$ and $a>0$,  
	$$\min_{|U|=a}\left\{\int_{U} |Du|^p \, dx \colon   u\in W^{1,\,p}_0(U), \|u\|_{L^p(U)}=1\right\}. $$
	According to the Rayleigh-Faber-Krahn inequality, it can be deduced that this quantity is not inferior to the corresponding Dirichlet eigenvalue of a Euclidean ball with a volume of 
	$a$. Subsequent research, particularly through transportation techniques as explored in \cite{MV2005, MV2008}, has revealed that  \emph{balls have the worst best Sobolev inequalities}. To be more specific, for any locally Lipschitz open domain $\Omega$ in $\mathbb{R}^n$ and $1\le p<n$, we define 
    $$\Phi_{\Omega}^{(p)}(T):=\inf\{\|\nabla f\|_{L^p(\Omega)}:\| f\|_{L^{p^*}(\Omega)}=1,\,\| f\|_{L^{p^\#}(\partial\Omega)}=T,\, f\in L_{loc}^1(\Omega)\text{ with }\lim_{x\to \infty}f=0\},$$
    where $p^*=\frac{np}{n-p}$ and $p^{\#}=\frac{(n-1)p}{np}$. Then the unit ball $B$ has the lowest $\Phi$-curve in the following sense:
    $$\Phi_{\Omega}^{(p)}(T)\ge \Phi_{B}^{(p)}(T)\quad \text{on }[0, T_n(p)],$$
    where $T_{n}(p):=\left(n|B|^{1/n}\right)^{1/p^{\#}}$ and $|B|$ is the Lebesgue measure of $B$. 
      Additional insights and recent advancements in this domain can be found in \cite{MN2023} and its associated references.

	Conversely, a distinctive category of domains, termed as \emph{John domains}, supports for Sobolev-Poincar\'e inequalities.
	A (bounded) domain $\Omega\subset \mathbb R^n$ is $J$-John for some $J\ge 1$ if there exists a distinguished point $x_0\in \Omega$ so that, for any $x\in \Omega$, one can find a curve $\gamma\subset \Omega$ joining $x$ to $x_0$ satisfying 
	$$\ell(\gamma[x,\,y])\le J d(y,\,\partial \Omega) \ \text{ for each } y\in \gamma, $$
	where $\gamma[x,\,y]$ is the subcurve of $\gamma$ joining $x$ and $y$. The constant $J$ is usually called the \emph{John constant}. 
	Heuristically speaking,  $\Omega$ contains a uniformly  linearly opened twisted cone  at every $x\in \Omega$; see Figure~\ref{fig:John}. 
	Standard examples of John domains encompass Lipschitz domains in any $\mathbb R^n$, and quasidisks in the plane, which include von Koch's snowflakes, see e.g.\ \cite[Chapter 6]{GH2012}. Stemming from the definition of a John domain and the Lebesgue differentiation theorem, it can be deduced that the boundary of a John domain possesses a Lebesgue measure of $0$.
	
	\begin{figure}[ht]
		\centering
		
		\def\svgwidth{\columnwidth}
		\resizebox{0.7\textwidth}{!}{
\begingroup%
  \makeatletter%
  \providecommand\color[2][]{%
    \errmessage{(Inkscape) Color is used for the text in Inkscape, but the package 'color.sty' is not loaded}%
    \renewcommand\color[2][]{}%
  }%
  \providecommand\transparent[1]{%
    \errmessage{(Inkscape) Transparency is used (non-zero) for the text in Inkscape, but the package 'transparent.sty' is not loaded}%
    \renewcommand\transparent[1]{}%
  }%
  \providecommand\rotatebox[2]{#2}%
  \newcommand*\fsize{\dimexpr\f@size pt\relax}%
  \newcommand*\lineheight[1]{\fontsize{\fsize}{#1\fsize}\selectfont}%
  \ifx\svgwidth\undefined%
    \setlength{\unitlength}{209.76377106bp}%
    \ifx\svgscale\undefined%
      \relax%
    \else%
      \setlength{\unitlength}{\unitlength * \real{\svgscale}}%
    \fi%
  \else%
    \setlength{\unitlength}{\svgwidth}%
  \fi%
  \global\let\svgwidth\undefined%
  \global\let\svgscale\undefined%
  \makeatother%
  \begin{picture}(1,0.35754506)%
    \lineheight{1}%
    \setlength\tabcolsep{0pt}%
    \put(0,0){\includegraphics[width=\unitlength,page=1]{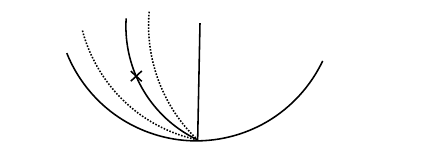}}%
    \put(0.32715372,0.18324185){\makebox(0,0)[lt]{\lineheight{1.25}\smash{\begin{tabular}[t]{l}$y$\end{tabular}}}}%
    \put(0.45855153,0.05452563){\makebox(0,0)[lt]{\lineheight{1.25}\smash{\begin{tabular}[t]{l}$x$\end{tabular}}}}%
  \end{picture}%
\endgroup%
}
		\caption{A  domain $\Omega$ is John if, heuristically speaking,  it contains a uniformly  linearly opened twisted cone  at every $x\in \Omega$. }
		\label{fig:John}
	\end{figure}

	For a domain $\Omega\subset \mathbb R^n$ supporting a $(p,p^*)$-Sobolev-Poincar\'e inequality for  $1\le p<n$, it implies that, for every $u\in W^{1,\,p}(\Omega),$ one has
	$$\inf_c \left(\int_{\Omega} |u-c|^{p^*}\, dx\right)^{\frac 1 {p^*}}\le C(n,\,p,\,\Omega) \left(\int_{\Omega} |Du|^{p}\,dx \right)^{\frac 1 {p}},$$
	where $p^*=\frac {np}{n-p}$ is the critical Sobolev exponent. 
	For comprehensive studies on the  Sobolev-Poincar\'e inequality, we recommend consulting \cite{B1988} and \cite{GR1983}. Furthermore, for an exploration of this inequality in the context of general metric measure spaces, encompassing Carnot groups, \cite{HK2000} serves as a valuable resource. 
	Moreover, in the specific context of John domains, which inherently support a Poincar\'e inequality, one can also establish trace inequalities for Sobolev functions with additional assumptions. For example,  a type of Poincar\'e inequality \cite[Theorem 4.4]{BS2007} holds when the domain is inner uniform (note that an inner uniform domain is John). Then the results in  \cite{M2017}, \cite{LS2018} and \cite{MSS2018} yield the desired trace inequalities in this domain.
	
	In contrast, Buckley and Koskela, as shown in \cite{BK1995}, revealed that a domain 
	$\Omega\subset\mathbb R^n$, possessing finite volume and adhering to a ball-separation property, supports Sobolev-Poincar\'e inequalities. This characteristic is particularly evident in scenarios involving conformal deformations, as discussed in \cite{BHK2001, BB2003}, such as bounded and simply connected domains in the plane.  The implications of this discovery underscore a deep connection between shape optimization problems involving Sobolev-Poincar\'e inequalities and the concept of John domains.
	
	This correlation prompts the need for a refined definition of the John constant, one that can be extended to arbitrary Euclidean domains. The ensuing definition is motivated by this imperative, and it is formulated to accommodate general Minkowski norms (defined at the beginning of Section~\ref{lower John}) in 
	$\mathbb R^n$  for potential applications in some other forthcoming research endeavors.

    \subsection{General Minkowski norm.}

In a recent manuscript \cite{SZ2024}, the authors presented an alternative proof of the seminal result obtained by Figalli, Maggi, and Pratelli \cite{FMP2010}, on the stability of isoperimetric inequality with respect to a general Minkowski norm, utilizing the John property of (almost) minimal surfaces. Partially motivated by this work, we consider John domains within the context of a general Minkowski norm in the first part of our manuscript.
 
    Some basic notations need to be clarified here. A function 
	$$\|\cdot\|\colon \mathbb R^n\to \mathbb R_+$$
	is a general Minkowski norm if it satisfies
	$$\|x+y\|\le \|x\| + \|y\|, \quad \forall x,\,y\in \mathbb R^n,$$
	$$\|\lambda x\|= \lambda \|x\|, \quad \forall x\in \mathbb R^n,\,\lambda>0,$$
	and
	$$\|x\|=0\  \text{ if and only if } \ x=0;  $$
	see e.g.\ \cite[Section 2.1]{ANP2002}. 
	Specifically, the standard Euclidean norm is denoted by $|\cdot| $. 
    Naturally, there exists a convex body
    $$\mathcal K_{\|\cdot\|}:=\{x\in\mathbb R^n:\|x\|< 1\}$$
   associated to $\|\cdot\|$.
   
   For a non-empty open set $\Omega\subsetneqq\mathbb{R}^n$ and $x\in\Omega$, we denote by $\partial\Omega$  the topological boundary of $\Omega$. We write $$d_{\Vert \cdot\Vert}(x,\partial \Omega):=\inf_{y\in \partial \Omega} \Vert x-y\Vert, $$
	and when the norm is the standard Euclidean one, we simply write
	$$d(x,\partial \Omega):=\inf_{y\in \partial\Omega}|x-y|.$$ 
	For $x\in\mathbb{R}^n$ and $r>0$, we use the notation
	$$B_{\Vert \cdot\Vert}(x,r):=\{y\in\mathbb{R}^n:\Vert x-y\Vert< r\}$$
	and by $\overline{B}_{\Vert\cdot\Vert}(x,r)$ its closure. 
	We drop the subindices and write 
	$B(x,r)$
	when the norm is the standard Euclidean norm. Especially, we denote the ball $B_{\|\cdot\|}(0,r)$ centered at $0$ by $B_{\|
    \cdot\|,r}$ for brevity.
    	
	Suppose that $(X,\Vert \cdot\Vert)$ is  a general Minkowski space and $\gamma\subset X$ is a rectifiable curve. Using reparametrization, $\gamma$ can be seen as a homeomorphism 
	$$\gamma:[0,1]\to X,\qquad t\mapsto \gamma(t).$$
	For every two distinct points $a_1,a_2\in \gamma$, there exists $t_1,t_2\in [0,1]$, such that $a_i=\gamma(t_i)$ for $i\in \{1,2\}$. We may assume $t_1<t_2$. Then we denote the subcurve $\gamma([t_1,t_2])$  joining $a_1$ to $a_2$ by $\gamma [a_1,a_2]$. Under the assumption above, 
	the length of the rectifiable curve $\gamma\subset X$ is written as 
	$$\ell_{\|\cdot\|}(\gamma)=\sup\left\{\sum_{i=0}^{N-1}\Vert \gamma(t_{i+1})-\gamma(t_{i})\Vert:0=a_0<a
	_1<a_2<\cdots<a_N=1, \quad N\in\mathbb{N}^+\right\}.$$ 
	If $\gamma$ is the union of curves, then $\ell_{\|\cdot\|}(\gamma)$ denotes the sum of the length of these curves under the same parametrization.

    \begin{defi}\label{def John}
        For a general Minkowski norm $\|\cdot\|$ and $J\ge 1$, a (bounded) domain $\Omega\subset \mathbb R^n$ is $J$-John  if there exists a distinguished point $x_0\in \Omega$ so that, for any $x\in \Omega$, one can find a curve $\gamma\subset \Omega$ joining $x$ to $x_0$ satisfying 
	$$\ell_{\|\cdot\|}(\gamma[x,\,y])\le J d_{\|\cdot\|}(y,\,\partial \Omega) \ \text{ for each } y\in \gamma.$$
    \end{defi}

Set
	\begin{equation}\label{non-balance}
		C_{\|\cdot\|}:=\max_{\|x\|=1}\|-x\|.
	\end{equation}
We emphasize here that the value of $C_{\|\cdot\|}$
  plays a crucial role in determining whether $\|\cdot\|$ constitutes a norm, as well as in influencing the length of the curve.
 \begin{rem} \label{para direction}
   When $C_{\|\cdot\|}=1$, $\|\cdot\|$ satisfies the properties of a norm. Conversely, if $C_{\|\cdot\|}\neq1$, the convex body $\mathcal K_{\|\cdot\|}$ associated to $\|\cdot\|$  loses its symmetry relative to the origin. In such instances, the length of curves becomes dependent on their parametrized direction. 
   
   A straightforward illustration is the case where, for some point $x_0\in \mathcal K_{\|\cdot\|}$ with $-x_0\notin \mathcal K_{\|\cdot\|}$. Then  the length of the line segment  parametrized from $0$ to $x_0$ is smaller than $1$, while the one in the reverse direction is  larger than $1$.
 \end{rem}

	\begin{defi}\label{def J}
		Consider the Euclidean space $(\mathbb R^n,\,\|\cdot\|)$ endowed with a general Minkowski norm  $\|\cdot\|$, and let $\Omega\subset \mathbb R^n$ be a (bounded) domain. Then  
		for any $x\in \Omega$ and  a curve $\gamma\subset\Omega$  containing $x$ and parametrized\footnote{We employ the standard abuse of notation here, using the same symbol for both the map and its image of a curve.} as $\gamma\colon [0,\,1]\to \Omega$, we define a function
		$j(\cdot;x,\gamma,\Omega):[0,1] \to \mathbb{R}$ as 

  $$j(t;x,\gamma,\Omega):=\frac{\ell_{\|\cdot\|}(\gamma([0,\,t]))}{d_{\Vert \cdot \Vert}(\gamma(t),\partial\Omega)}\quad \text{ for any } t\in [0,1].$$
		
		Subsequently, we set
		$$
		J(x,\Omega;x_0):=\inf_{\beta\subset \Omega}\left\{ \sup_{t\in [0,1]}j(t\,;x,\beta,\Omega):   \beta\subset\Omega  \text{  is a curve joining $x$ to $x_0\in \Omega$}  \right\}, 
		$$
		and 
		$$J(\Omega;x_0):=\sup_{x\in\Omega}J(x,\Omega;x_0). $$ 
		We say that 
		$\Omega$ satisfies the $J$-carrot John condition with center $x_0\in\Omega$ if 
		$$J=J(\Omega;x_0) <\infty.$$

		We define  $John(\cdot)$ on the collection of bounded domains of $\mathbb{R}^n$ as $$John(\Omega):=\inf_{x_0\in\Omega}\{J(\Omega;x_0)\},$$ 
		and designate $John(\Omega)$ as the 
		(optimal) John constant  of $\Omega$. 
	\end{defi}

	By the definition of $John(\cdot)$ and the definition of John domain, it follows that $\Omega\subset\mathbb{R}^n$ is a John domain with center $x_0\in \Omega$ if and only if 
	$John(\Omega)<+\infty.$
	
	In the pursuit of broader applications, we extend the definition of the $J$-carrot to a pair of suitable points 
	$x,\, x_0\in \dot {\mathbb R}^n$, where $\dot {\mathbb R}^n$ represents the one-point compactification of $\mathbb R^n$. 
	\begin{defi}\label{carrot def}
		Let $x\in \mathbb R^n$ and $x_0 \in   \dot{\mathbb{R}}^n$  be distinct points and $\gz\subset  {\mathbb{R}}^n$ be a curve joining $x$ toward $x_0$. Assume that  $J\ge 1$.  
		
		When $x_0\neq \infty$, we define
		$$car(\gz, J):=\bigcup\big\{B_{\Vert \cdot\Vert}(y,\ell_{\|\cdot\|}(\gamma[x,y])/J): y \in \gamma \setminus\{x\}\big\},$$
while when $x_0=\infty$, we define
		$$car(\gz, J):=\bigcup\big\{B_{\Vert \cdot\Vert}(y,\ell_{\|\cdot\|}(\gamma[x,y])/J): y \in \gamma \setminus\{x,\infty\}\big\}.$$
		Then the (open) set $car(\gamma, J)$ is called the $J$-$carrot$, with core $\gamma$ and vertex $x$, joining $x$ to $x_0$.

		We say an open set  $\Omega\subset {\mathbb R}^n$ satisfies $J$-carrot John condition  with center $x_0 \in \Omega\cup \{\infty\}$, if for each  point  $x\in \Omega$, there exists a curve $\beta\subset \Omega$ joining $x$ toward $x_0$ so that $car(\beta, J)\subset \Omega$. Furthermore, if $\Omega$ also satisfies connectivity, we say that $\Omega$ is a $J$-carrot John domain.
	\end{defi}

	\begin{rem}
		It is noteworthy that  in the definition of $car(\gamma,\,J)$, one has the flexibility to substitute $\ell_{\|\cdot\|}(\gamma[x,\,y])$
		by either $\diam_{\|\cdot\|}(\gamma[x,\,y])$ or simply $\|y-x\|$.
		Importantly, these alternative formulations are equivalent in both bounded and unbounded scenarios, as elucidated in, for instance, \cite[Theorem 2.14]{NV1991}.
	\end{rem}
	
	\begin{rem}\label{cigar defi}
		In the literature, an alternative definition of the John domain is sometimes employed, where the term ``$J$-carrot" is replaced by the so-called ``$J$-cigar". To elucidate, when considering a pair of distinct points $x,\,y\in \mathbb{R}^n$   and a curve   $\beta\subset\mathbb{R}^n$ containing $x$ and $y$, the ``$J$-cigar" is defined as:
		$$cig(\beta,J):=\bigcup\big\{B_{\Vert \cdot\Vert}(\eta,\rho(\eta)/J):\eta\in \beta\setminus\{x,y\}\big\},$$
		where
		$$\rho(\eta)=\min\{\ell_{\|\cdot\|}(\beta[x,\eta]),\ell_{\|\cdot\|}(\beta[y,\eta])\}.$$
		The set $cig(\beta,J)$ is called the $J$-cigar with core $\beta$ joining $x$ and $y$, and $\Omega$ is $J$-cigar John if each pair of points $x,\,y\in \Omega$ can be joined by  a curve  $\beta\subset \Omega$ satisfying  $cig(\beta,J)\subset \Omega$.
		Heuristically speaking, in the bounded case, one can interpret a $J$-cigar as the union of two $J$-carrots. Indeed, it has been rigorously established that when 		$\Omega\subset \mathbb R^n$		is bounded, these two definitions, employing either the $J$-carrot or the $J$-cigar, are equivalent; refer to, for example, \cite[Theorem 2.16]{NV1991}, and also  Lemma~\ref{cigar repla} in the manuscript. In addition, for a discussion of the unbounded case, see Remark 1.8.
	\end{rem}
	
	\subsection{Bounded John domains}	Now we are prepared to articulate our first theorem.
	
	\begin{thm}[lower-semicontinuity of (optimal) John constants]\label{low} Let $J_0\ge 2$ and assume that $\{\Omega_j\}_{j\in\mathbb{N}^+}$ is a sequence of uniformly bounded  John domains satisfying 
		$$John(\Omega_j)\le J_0 \ \text{ and } \ \vert \Omega_j\vert\ge c_0\vert B_{\Vert \cdot\Vert}(0,1)\vert,$$ for some $c_0> 0$. 
		Then up to passing to a subsequence, $\overline{\Omega}_j$ converges to some compact set $A\subset \mathbb R^n$ in the Hausdorff distance $d_H$ so that the interior $\Omega$ of $A$ satisfies 
		\begin{enumerate}
			\item [(i)]  $\max_{x\in\Omega}d_{\Vert \cdot\Vert}(x,\partial\Omega)\ge c=c(n,\, C_{\|\cdot\|},\,J_0,\,c_0)>0$, where $C_{\|\cdot\|}$ is defined in \eqref{non-balance}. 
			\item [(ii)]$\Omega $  is a  John domain with 
			$$John(\Omega)\le \liminf_{j\to \infty}{John(\Omega_j)}.$$
		\end{enumerate}	
	\end{thm}
	
	In Theorem~\ref{low}, one can only anticipate lower semicontinuity, not continuity. To illustrate, consider the sequence of sets
	$$\Omega_k:=\mathbb D\setminus [0,\,1]\times [-2^{-k},\,2^{-k}], \quad k\ge 1$$
	where $\mathbb D$ denotes the unit disk in the plane. Then $John(\Omega_k)$ is uniformly bounded below, away from  $1$, while the limit of $\overline{\Omega}_k$ is $\overline{\mathbb D}$ as $k\to \infty$, whose interior has an (optimal) John constant of 
	$1$.

	\begin{cor}
		For  $R\gg |D|$ satisfying that $D\subset B_{\|\cdot\|,\,R}$, 
		$$\min\left\{John(\Omega)\colon |\Omega|= |D|,\  \Omega\subset B_{\|\cdot\|,\,R}\right\} $$
		has a solution, where $D=-\mathcal{K}_{\|\cdot\|}$. 
		Moreover, the set of minimizers precisely consists of translations of	$D$. 
	\end{cor}
	\begin{proof}
		Let $\Omega_k$ be a minimizing sequence. 
		As $\partial \Omega_k$ has Lebesgue measure $0$, then 
		$$|\Omega_k|=|\overline{\Omega}_k|.$$
		Then as a direct consequence of Theorem~\ref{low}, up to passing to a subsequence, $\overline{\Omega}_k\to \overline{\Omega}$ for some open set $\Omega\subset  B_{\|\cdot\|,\,R}$, together with 
		$$John(\Omega)\le \liminf_{k\to \infty}{John(\Omega_k)}\le J$$
		for some $J\ge 1$.
		Moreover, by \cite[Theorem 2.8]{V2000}\footnote{Even though in \cite[Theorem 2.8]{V2000} it is only proved that for a $J$-carrot John domain $U\subset \mathbb R^n$ with $diam(U)\le 1$,  
			$$|\{x\in U:d(x,\partial U)<t\}|\le \mu(t,J,n)\to 0\qquad\text{ as }t\to 0.$$
			However, it follows from a similar argument that for a bounded $J$-carrot John domain $U\subset \mathbb R^n$ with $|U|\le M$, where $M$ is a positive constant,  
			$$|\{x\in \mathbb R^n: d(x,\partial U)<t\}|\le \mu(t,J,n,M)\to 0\qquad\text{ as }t\to 0.$$ 
			This, coupled with the fact that $\overline{\Omega}_k$ forms a Cauchy sequence in terms of the Hausdorff distance, leads us to the desired conclusion.}, Lebesgue measure is  continuous with respect to the Hausdorff metric for $J$-carrot John domains. 
		Thus $\Omega$ is a desired minimizer.

		Now we show that a minimizer $\Omega$ must be a translation of $D$. Indeed, since $John(\Omega)\ge 1$ and $John(D)=1$, then it follows that $John(\Omega)= 1$. Now by the definition of $John(\Omega)$ and Lemma~\ref{continuous 2}, saying that the infimum of $x_0$ is taken away from the boundary, we conclude that for any $y\in \Omega$
		$$\|y-x_0\|\le \ell_{\|\cdot\|}(\gamma_{y,\,x_0})\le d_{\|\cdot\|}(x_0,\,\partial\Omega),$$
		where $\gamma_{y,\,x_0}$ is a  John curve joining $y$ to $x_0$ given  by Lemma~\ref{J eqiup}.  
		Thus $\Omega$ is  a translation of $D$.
	\end{proof}

	We expect that this outcome is intricately connected to the observation that ``balls have the worst best Sobolev inequalities." 	
	In contrast, it was proven that a Jordan domain $\Omega\subset \mathbb R^2$ qualifies as a quasidisk if and only if both $\Omega$ 
	and its complementary domain are John domains, as documented in \cite{NV1991} and \cite[Theorem 6.12]{GH2012}. Consequently, considering the role of normalized quasidisks in modeling the universal Teichm\"uller space \cite[Section III.1.5]{L1987}, Theorem~\ref{low} not only enables the exploration of extremal maps in quasiconformal mappings but also offers insights into the properties of quasidisks.


	\subsection{Unbounded open sets satisfying carrot John condition with center $\infty$}

	Advancing in our exploration, we turn our attention to the $J$-carrot John condition for unbounded domains $\Omega\subset \mathbb R^n$ (with unbounded $\partial \Omega$). Namely, for any $x\in \Omega$, there exists a curve $\gamma\subset \Omega$ from $x$ toward $\infty$  in such a way that the infinite $J$-carrot 
	$$car(\gamma,\,J)\subset \Omega.$$ 
	Such domains find relevance in the exploration of the Mumford-Shah problem, as expounded in, for instance, \cite[Section 19]{BD2001} and \cite[Section 56, Proposition 7]{D2006}. Also see \cite{Z2025} for the application of (a local version of) the following theorem.

	\begin{thm}\label{John component}
		
		Suppose that  $K\subset \mathbb R^n$ is a closed set, $\mathbb{R}^n\setminus K $ is an unbounded open set  satisfying the $J$-carrot John condition with  center $\infty$ and $0\in K$. Then for any 
		$R\ge 0$, there exist  at most $N$-many $J'$-carrot John subdomains (where some of them could be empty) $\{W_{j,R}\}_{j\in \{1,\cdots,N\}}$ of $\mathbb{R}^n\setminus K$ of $\mathbb{R}^n$ with $J'=J'(n,\, J)$ and $N=N(n,\, J)$, such that
		
		\begin{enumerate}[(i)]
			
			\item We have 
			$$B_{R}\setminus K\subset   \bigcup_{j=1}^N  \overline{W}_{j,\,R}, $$
			together with $W_{j,\,R}\subset B_{C'R}$ with $C'=C'(n,\,J)$ and (if it is non-empty)
			\begin{equation}\label{volumn relat}
				C(n,\, J)^{-1}R^n\le  |W_{j,R}|\le C(n,\, J)R^n.
			\end{equation}
			
			In addition, for each $1\le k\le N$ and $R>0$, there exists a sequence $\{k_l\}_{l=0}^{+\infty}$ and $k_0=k$ so that
			\begin{equation}\label{volumn relat 2}
				C(n,\, J)^{-1}|W_{k_{l},2^{l}R}|\le  |W_{k_{l},2^{l}R}\cap W_{k_{l+1},2^{l+1} R}|, \quad l\ge 0.
			\end{equation}

			\item 
			For $1\le j\le N$ and some $x_j\in \mathbb R^n,$ the set
			$$W_{j,\infty}:=\bigcup_{R>|x_j|} W_{j,R}\subset \mathbb{R}^n\setminus K$$
			is also a $J'$-carrot John subdomain centered at $\infty$, for which 
			\begin{equation}\label{whole cover}
				\mathbb{R}^n\setminus K\subset\bigcup_{j=1}^{N}\overline{W}_{j,\infty}. 
			\end{equation}
			Moreover, for any $z,w\in W_{j,\infty}$, there exists a ball $B_{z,w}\subset W_{j,\infty}$ whose  radius is $r_{z,w}$ so that there are two rectifiable curves
 $\gamma_z,\gamma_w$ respectively joining $z,w $ to the center $a_{z,w}$ of $B_{z,w}$ satisfying 
			\begin{equation}\label{Boman chain}
				B_{z,w}\subset car(\gamma_z,J')\subset W_{j,\infty}\quad\text{and}\quad  B_{z,w}\subset car(\gamma_w,J')\subset W_{j,\infty},
			\end{equation}
			where  the radius $r_{z,w}$ satisfies
			\begin{equation}\label{cigar ball}
				\frac{\ell(\gamma_z[z,a_{z,w}])}{J'}=r_{z,w}=\frac{\ell(\gamma_w[w,a_{z,w}])}{J'}.
			\end{equation}

			\item 	In particular, as a consequence of \cite{B1988}, \cite{GR1983} (for bounded domains) together with \cite{H1992} (for unbounded domains), we have
			$$\inf_c \left(\int_{W_{j,\,R}} |u-c|^{p^*}\, dx\right)^{\frac 1 {p^*}}\le C(n,\,p,\,J) \left(\int_{W_{j,\,R}} |Du|^{p}\,dx \right)^{\frac 1 {p}} \  \text{ for any $u\in W^{1,\,p}(W_{j,\,R})$},$$
			and
			$$\inf_c \left(\int_{W_{j,\,\infty}} |u-c|^{p^*}\, dx\right)^{\frac 1 {p^*}}\le C(n,\,p,\,J) \left(\int_{W_{j,\,\infty}} |Du|^{p}\,dx \right)^{\frac 1 {p}}\ \text{ for any $u\in W^{1,\,p}(W_{j,\,\infty})$}.$$
		\end{enumerate}	
	\end{thm}
	
	\begin{rem}
		As noted in Remark \ref{cigar defi}, according to \cite[Theorem 2.16]{NV1991}, the $J$-carrot John condition and the $J$-cigar John condition are equivalent for any bounded domain, up to positive constants.
		
		However, this equivalence does not necessarily hold for unbounded domains, and the Sobolev-Poincar\'e inequality  in  \cite{H1992} is proven for unbounded cigar John domains.
		An example for the failure of the equivalence is given by the following: 	
		Consider the unbounded domain
		$$U=\mathbb{R}^2\setminus \big((-\infty,-1]\times\{0\}\cup[1,+\infty)\times\{0\}\big)$$
		which satisfies the $1$-carrot John condition with center $\infty$. However, it does not satisfy any cigar John condition. Nevertheless, observe that,   $U$ can be covered as the union of two sets $\mathbb{H}^+\cup B(0,1)$ and $\mathbb{H}^-$, where $\mathbb H^{\pm}$ denote the upper/lower (open) half plane, and  each of them individually satisfies the $2$-cigar John condition. 
		
		In a similar vein, Theorem~\ref{John component} establishes that any unbounded $J$-carrot John domain can be covered by a uniformly finite number of $J'$-cigar John domains, where the number of domains is uniformly bounded depending only on $J$ and $n$. 
		We remark that \eqref{Boman chain} is indeed equivalent to stating that every two distinct points $z,\,w$ can be connected by  a $J'$-cigar inside $W_{j,\,\infty}$.  However, to streamline terminology, the theorem is presented in the context of carrots.
	\end{rem}

	The manuscript is structured as follows: In Section~\ref{lower John}, we provide the proof of Theorem~\ref{low}, devoting careful attention to the continuity of functions as defined in Definition~\ref{def J}. A pivotal lemma, namely Lemma~\ref{carrot convergent}, examines the behavior of carrots under Hausdorff convergence. Another crucial aspect involves preventing the John centers of converging John domains from reaching the boundary, a concern addressed in Lemma~\ref{continuous 2}.
	
	The proof of Theorem~\ref{John component} is detailed in Section~\ref{unbdd carrot John}, with an introductory overview of the proof presented at the outset of the section.
	
	

	\section{Lower-semicontinuity of John constant}\label{lower John}
	In our manuscript, we employ the notation as follows: For any set $E\subset\mathbb{R}^n$, the closure of $E$ with respect to the Euclidean topology is denoted as $\overline{E}$ or $Cl(E)$, and its complement is denoted by $E^c$. Given that the Euclidean space is of finite dimension, the topology induced by the norms remains the same.

	The space consisting of all nonempty compact sets in $\mathbb{R}^n$ equipped with the Hausdorff metric $d_H$ is denoted as $(\mathcal{C}^n, d_H)$. The topologies of $(\mathcal{C}^n, d_H)$ induced by all norms in $\mathbb R^n$ are equivalent. For simplicity, one can think of   $d_H$ as the metric induced by the standard Euclidean norm in $\mathbb R^n$.

	The Lebesgue measure of the set $E\subset\mathbb{R}^n$ is denoted by $\vert E\vert$ and the $s$-dimension Hausdorff measure of $E$ is denoted by $\mathcal{H}^{s}(E)$. A general constant is denoted by 
	$C$, which may vary across different estimates, and we include all the constants it depends on within the parentheses, denoted as $C(\cdot)$.
	


	We next prove a key lemma, which later helps us to deduce the lower-semicontinuity of both the function $J(\Omega;\cdot):\Omega\to [1,+\infty)$ and the (optimal) John constant $John(\cdot)$. 
	\begin{lem}\label{carrot convergent}
		Let $\{x_i\}_{i\in\mathbb N}$ and $\{y_i\}_{i\in\mathbb N}$ be two sequences of points with $x_i\in\mathbb R^n$ and $y_i\in\dot{\mathbb R}^n$ for $i\in\mathbb N$. Assume that $\{\gamma_i\}_{i\in\mathbb{N}}$ is a sequence of locally rectifiable curves   in ${\mathbb{R}}^n$ joining pairs of distinct points $x_i,\,y_i$. 
		$$\lim_{i\to+\infty}x_i=:x\neq\infty \ \ \text{ and } \ \ y:=\lim_{i\to+\infty}y_i$$
		exist in $\dot{\mathbb R}^n$ and that		 either
		$$ \ell_{\|\cdot\|}(\gamma_i) \text{ is uniformly bounded,}$$
		or
		\begin{equation}\label{uniform finite length}
			y=\infty \quad \text{ and } \quad \ell_{\|\cdot\|}(\gamma_i\cap B_R),\ R\ge 1  \ \text{ uniformly bounded independent of $i$} 
		\end{equation}
		holds. 
		
		Moreover,  let $\{J_i\}_{i\in\mathbb{N}}$ be a uniformly bounded sequence with $J\ge 1$ and $car(\gamma_i,J_i)$ is the corresponding $J_i$-carrot joining $x_i$ toward $y_i$, respectively. Then up to relabeling the sequence,
		\item [(i)] In $(\mathbb R^n,\|\cdot\|)$
		\begin{equation}\label{hausdorff con}
			\gamma_i \to \gamma \quad \text{ locally uniformly, }
		\end{equation}
		and
		\begin{equation}\label{car lemm}
			car( \gamma,J)\subset \bigcap_{m=1}^{+\infty}{\bigcup_{i=m}^{+\infty}car(\gamma_i,J_i)},
		\end{equation}
		where $J:=\liminf_{i\to +\infty}J_i$. 
		\item [(ii)] 
		If $\ell_{\|\cdot\|}(\gamma_i)$ is uniformly bounded, then
		\begin{equation}\label{norm}
			\ell_{\|\cdot\|}( \gamma)\le \liminf_{i\to+\infty}\ell_{\|\cdot\|}(\gamma_i).
		\end{equation}
		
	\end{lem}
	
	\begin{proof}
		\noindent {\bf Case 1:   $\ell_{\|\cdot\|}(\gamma_i)$ is uniformly bounded}. 	Then our assumption implies that, for some $L>0$, 
		$$l_i:= \ell_{\|\cdot \|} \big(\gamma_i\big)\le L.$$ 
		As $\{J_i\}_{i\in\mathbb N}$ is a uniformly bounded sequence, we may assume   \begin{equation}\label{limit length}
			l_{\infty}:=\lim_{i\to+\infty}l_i,\quad  J:=\lim_{i\to+
				\infty} J_i
		\end{equation}
		with $l_{\infty}\le L$.

		By parameterizing $\gamma_i$ via arc length on $[0,\,L]$, up to extending as a constant curve if necessary on the interval $[\ell_{\|\cdot \|}(\gamma_i),\,L]$,  we obtain that $\{\gamma_i(\cdot)\}_{i\in\mathbb N}$ is equicontinuous and uniformly bounded as $x_i\to x\neq \infty$. Thus,  up to passing to a subsequence, it follows from the Arzel\'a-Ascoli theorem that, up to extracting a subsequence, 
		\begin{equation}\label{A-A first}
			\gamma_i\to \gamma\in C([0,L]; \mathbb{R}^n) \ \text{ uniformly.}
		\end{equation}
		As $\gamma_i$ is $1$-Lipschitz, $\gamma$ is also $1$-Lipschitz, and thus 
		\begin{equation}\label{y finite 2}
			\ell_{\|\cdot \|}(\gamma)\le \liminf_{i\to \infty}\ell_{\|\cdot \|}(\gamma_i)
		\end{equation}
	as desired.


		In addition, for any point $\zeta\in car( \gamma,J)$, there exists $t_z\in (0,L]$ with $z=\gamma(t_z)\in \gamma$
  so that		\begin{equation}\label{san jiao 1}
			\zeta\in B_{\Vert \cdot\Vert}( {z},\ell_{\|\cdot\|}( \gamma ([0,t_z])  )/J),
		\end{equation}
		which yields a positive constant $\delta:=\ell_{\|\cdot\|}(\gamma ([0,t_z]) )/J-\|z-\zeta\|>0$.		
		
		Note that \eqref{A-A first} yields the existence of a sequence $\{z_i\}_{i\in\mathbb{N}}$ for which $z_i=\gamma_i(t_z)\in \gamma_i$ and $z_i\to z$ as $i\to \infty$. Therefore, by \eqref{y finite 2}, for any positive $\epsilon<\delta/2$, when $i\ge i_0$ for some big integer $i_0$, we have $\|z_i-z\|<\epsilon$ and 
		$$\ell_{\|\cdot\|}(\gamma ([0,t_z]) )\le \ell_{\|\cdot\|}(\gamma_i([0,t_z]) )+\epsilon.$$
		As a result, combining \eqref{san jiao 1}  and the triangle inequality, the estimate above gives   
		\begin{align}
			\|z_i-\zeta\|&\le \|z-\zeta\|+\|z_i-z\|
			< \|z-\zeta\|+\epsilon \nonumber\\
			&\le (\ell_{\|\cdot\|}(\gamma ([0,t_z]) )/J-\delta)+\epsilon \le \ell_{\|\cdot\|}(\gamma_i ([0,t_z]) )/J+2\epsilon-\delta \nonumber\\
			& < \ell_{\|\cdot\|}(\gamma_i ([0,t_z]) )/J, \nonumber
		\end{align}
		so that 
		$$\zeta\in {\bigcup_{i=m}^{+\infty}B_{\Vert \cdot\Vert}\left( {z_i},\ell_{\|\cdot\|}(\gamma_i ([0,t_z]) )/J_i\right)}\subset {\bigcup_{i=m}^{+\infty}car(\gamma_i,J_i)}.$$  
		Consequently, $\zeta\in \bigcap_{m=1}^{+\infty}\Big({\bigcup_{i=m}^{+\infty}car(\gamma_i,J_i)}\Big)$, which implies   \eqref{car lemm}.
		In conclusion, when $y\neq\infty$, Lemma \ref{carrot convergent} holds.
		
		\noindent {\bf Case 2:  $y=\infty$ and $\ell(\gamma_i)\to \infty$}: 
		In this case, as $\gamma_i$ is locally rectifiable and satisfies
		\eqref{uniform finite length}, via suitable truncation, by Step 1 and applying a diagonal argument, we have $\gamma_i$ converges locally uniformly to a curve $\gamma$ parametrized by arc length on $[0,\,\infty)$. 
		Similarly, by taking the union of carrots along $\gamma_i$, \eqref{car lemm}  is obtained also from Step 1. 
	\end{proof}
	
	
		%

	For a bounded domain $\Omega$ and any rectifiable curve  $\gamma\subset\Omega$ joining $x$ to  $x_0$ with  $x,x_0\in \Omega$, 
	recall the definition of $j(t;x,\gamma,\Omega)$ in Definition~\ref{def J}. Then
	$j(t;x,\gamma,\Omega)$ is continuous   with respect to  $t\in [0,1]$, and then  the compactness of $[0,1]$ tells that  there exists a point $t_0\in [0,1]$ such that  $$\frac{\ell_{\|\cdot\|}(\gamma([0,t_0]))}{d_{\Vert \cdot \Vert}(\gamma(t_0),\partial\Omega)}=\sup_{t\in [0,1]}j(t;x,E,\Omega)<+\infty.$$	
	Thus, 
	$$J(x,\Omega;x_0):=\inf\left\{\sup_{t\in [0,1]}j(t;x,\beta,\Omega):   \beta\subset\Omega  \text{  is a curve joining $x$ to $x_0$}  \right\}$$ 
	is finite. We next show that  
	Lemma~\ref{carrot convergent} ensures the existence of the rectifiable curve who make this infimum  be reached.

	\begin{lem}\label{J eqiup}
		Assume $\Omega\subset \mathbb R^n$ is a bounded domain. Let $x,\,x_0 \in\Omega$ be two distinct points with $J(x,\Omega;x_0)<+\infty$. Then 
		there exists a rectifiable curve $\gamma \subset\Omega$ joining $x$ to $x_0$ such that 
		$$\sup_{t\in [0,1] }j(t;x,\gamma ,\Omega)=J (x,\Omega;x_0).$$
	\end{lem}
	
	\begin{proof}
		Choose a minimizing sequence $\{\gamma_i\}_{i\in\mathbb{N}^+},\,\gamma_i\subset \Omega$ of rectifiable curves joining $x$ to $x_0 $  so that  
		$$\lim_{i\to \infty} J_i:=\lim_{i\to \infty}\sup_{t\in [0,1]}j(t;x,\gamma_i,\Omega)=J(x,\Omega;x_0)=:J.$$
		Then by the Definition~\ref{def J}, the uniform boundedness of $J_i$ implies that  $\ell_{\|\cdot\|}(\gamma_i)$ is bounded uniformly. Thus 
		by letting $car(\gamma_i,J_i)$ be the $J_i$-$carrot$ joining $x$ to $x_0$ for $i\in\mathbb{N}$,  
		Lemma \ref{carrot convergent} tells that  there exists a rectifiable curve $\gamma $ of a $J$-$carrot$ joining $x$ to $x_0$, such that 
		$$car(\gamma,J) \subset\Omega,$$
		which implies that $J \ge \sup_{t\in 
 [0,1]}j(t;x,\gamma,\Omega) .$
		On the other hand,  the convergence of $J_i$ to $J$  together with Definition \ref{def J}, gives 
		$$J \le \sup_{t\in 
 [0,1]}j(t;x,\gamma ,\Omega) .$$		
		The proof is completed. 
		
	\end{proof}

	\begin{lem}\label{4c^2}
		Let $x_0\in\Omega$ and $\Omega\subset\mathbb{R}^n$ be a bounded $J$-carrot John domain with $J:=J(\Omega;x_0)$. Then for any $z\in\Omega$, $J(\Omega;z)$ is finite. Moreover, for any $y\in\Omega$ satisfying 
		$$d_{\Vert \cdot\Vert}(y,\partial\Omega)=\max_{x\in\Omega}d_{\Vert \cdot\Vert}(x,\partial\Omega),$$
		we get $J(\Omega;y)\le C(n, C_{\|\cdot\|},\,J)$.
	\end{lem}

	This lemma directly follows from Theorem 3.6 in \cite{V1994}. Despite their findings being initially formulated for the standard Euclidean norm, one can establish them for a general Minkowski norm in 
	$\mathbb R^n$
	by employing identical arguments, necessitating only notational adjustments, with additional dependency on 
	$ C_{\|\cdot\|}$\footnote{In the proof of  \cite[Theorem 3.6]{V1994}, by choosing $y$ as the center of the largest ball contained in $\Omega$, for any $x\in \Omega$, the John curve $\gamma$, as the core of $cig(\gamma, J)\subset \Omega$ joining $x$ to $y$, is proved to be the core of $car(\gamma, J_1)$ for some $J_1$. Due to the fact that $\mathcal K_{\|\cdot\|}$ might not be symmetric with respect to the origin, the upper bound estimate for $\ell_{\|\cdot\|}(\gamma[x,y])$ becomes $(1+C_{\|\cdot\|})d_{\|\cdot\|}(y,\partial\Omega)$, different from the Euclidean case. Consequently, $J_1$ further depends on $C_{\|\cdot\|}$; see also Remark \ref{para direction}.}.  
	
	\begin{lem}\label{continuous 1}
		Let $\Omega\subset \mathbb{R}^n$ be a bounded John domain.
		Then 
		$$J(\cdot,\Omega;\cdot): \Omega\times\Omega \to [1,+\infty),\quad (x,\,y)\mapsto J(x,\Omega; y)$$
		is   locally Lipschitz continuous. 
	\end{lem}
	
	\begin{proof}
		Given $(x,\,y)\in \Omega\times\Omega$ and $(\hat{x},\hat{y})\in \Omega\times\Omega$ close to $(x,\,y)$, 
		We  first estimate 
		$$J(\hat{x},\Omega;\hat{y})-J(x,\Omega;y)$$
		from above and below, respectively.
		
		\noindent {\bf Step 1: Estimate $J(\hat{x},\Omega;\hat{y})-J(x,\Omega;y)$ from above. }
		Let $$J:=J(x,\Omega;y).$$ 
		Then Lemma~\ref{J eqiup} yields a rectifiable curve $ \gamma\subset\Omega$ joining $x$ to $y$ together with the corresponding $J$-carrot $car(\gamma,J)$, such that 
		\begin{equation}\label{sp1}
			\sup_{t\in 
 [0,1]}j(t;x,\gamma,\Omega) =J(x,\Omega;y)=J\quad \text{and}\quad car(\gamma,J)\subset\Omega.
		\end{equation}
		As a consequence of the compactness of $[0,1]$,  the definition of $j(t;x,\gamma,\Omega)$ gives us a point $s\in [0,1]$,  such that
		
		\begin{equation}\label{sp2}
			\frac{\ell_{\|\cdot\|}(\gamma([0,s]))}{d_{\Vert\cdot\Vert}(\gamma(s),\partial\Omega)}=J=\sup_{t \in 
 [0,1]}j(t;x,\gamma,\Omega).
		\end{equation}

		We claim that
		\begin{equation}\label{sp1.5}
			d_{\Vert \cdot\Vert}(\gamma,\partial\Omega)\ge\frac{d_{\Vert\cdot\Vert}(x,\partial \Omega)}{2C_{\|\cdot\|}  J}.
		\end{equation}
		Indeed, for any $z\in \gamma\cap B_{\Vert\cdot\Vert}\left(x,\frac 1{2}{d_{\Vert\cdot\Vert}(x,\partial \Omega)}\right)$, the triangle inequality gives
		$$d_{\Vert \cdot\Vert}(z,\partial\Omega) \ge \frac 1 {2} {d_{\Vert\cdot\Vert}(x,\partial \Omega)};$$
		while for $z\in \gamma\setminus B_{\Vert\cdot\Vert}\left(x,\frac 1{2}{d_{\Vert\cdot\Vert}(x,\partial \Omega)}\right)$, it follows
		from  \eqref{sp1}  and the definition of $car(\gamma,J)$  that  
		$$ d_{\Vert\cdot\Vert}(z,\partial\Omega)\ge \frac{\ell_{\|\cdot\|}(\gamma[x,z])}{J}\ge \frac{\ell_{\|\cdot\|}(\gamma[z,x])}{C_{\|\cdot\|}J}\ge\frac {d_{\Vert\cdot\Vert}(x,\partial \Omega)}{2C_{\|\cdot\|}J}.$$ 
		As $J\ge 1$, our claim \eqref{sp1.5} follows.

		Let $(\hat{x},\hat{y})\in \Omega\times\Omega$ close to $(x,\,y)$ with 
		$$\hat{x }\in B_{\|\cdot\|}\left(x,\frac{ d_{\Vert\cdot\Vert}(x,\partial\Omega)}{2}\right) \quad \text{ and } \quad \hat{y}\in B_{\|\cdot\|}\left(y,\frac{ d_{\Vert\cdot\Vert}(y,\partial\Omega)}{2}\right). $$
		Set 
		\begin{equation}\label{gamma line}
			\hat{\gamma}\subset L_{\hat{x},x} \cup \gamma \cup L_{y,\hat{y}}, 
		\end{equation}
	be a rectifiable curve joining $\hat x$ to $\hat y$. where  $L_{\hat{x},x}$ is the line segment joining  
		$\hat{x}$ to $x$ and $L_{y,\hat{y}}$ is the one joining
		$y$ to $\hat{y}$. 
		As $y\in \gamma$, we conclude from \eqref{sp2} that
		$$\frac{\ell_{\|\cdot\|}(\gamma[x,y])}{d_{\Vert\cdot\Vert}(y,\partial\Omega)} \le \frac{\ell_{\|\cdot\|}(\gamma([0,s]))}{d_{\Vert\cdot\Vert}(\gamma(s),\partial\Omega)},$$
		and hence
		\begin{align}\label{sp8}
			& \frac{\ell_{\|\cdot\|}(\gamma[x,y])}{d_{\Vert\cdot\Vert}(y,\partial\Omega)-\Vert y-\hat{y}\Vert}-\frac{\ell_{\|\cdot\|}(\gamma([0,s]))}{d_{\Vert\cdot\Vert}(\gamma(s),\partial\Omega)}\nonumber \\
			&=\left(\frac{\ell_{\|\cdot\|}(\gamma[x,y])}{d_{\Vert\cdot\Vert}(y,\partial\Omega)-\Vert y-\hat{y}\Vert}-\frac{\ell_{\|\cdot\|}(\gamma[x,y])}{d_{\Vert\cdot\Vert}(y,\partial\Omega)} \right)+\left(\frac{\ell_{\|\cdot\|}(\gamma[x,y])}{d_{\Vert\cdot\Vert}(y,\partial\Omega)} -\frac{\ell_{\|\cdot\|}(\gamma([0,s]))}{d_{\Vert\cdot\Vert}(\gamma(s),\partial\Omega)}\right) \nonumber \\  
			&\le \frac{\Vert y-\hat{y}\Vert}{d_{\Vert\cdot\Vert}(y,\partial\Omega)(d_{\Vert\cdot\Vert}(y,\partial\Omega)-\Vert y-\hat{y}\Vert)} \ell_{\|\cdot\|}(\gamma[x,y]).   
		\end{align}

		For each $z\in \hat \gamma$, we now estimate 
		$$ \frac{\ell_{\|\cdot\|}(\hat{\gamma}[\hat{x},z])}{d_{\Vert\cdot\Vert}(z,\partial\Omega)}
		-\frac{\ell_{\|\cdot\|}(\gamma([0,s]))}{d_{\Vert\cdot\Vert}(\gamma(s),\partial\Omega)}$$
		in three cases. 
		
		First of all, when $z\in \gamma$, as
		$$\frac{\ell_{\|\cdot\|}(\gamma[x,z])}{d_{\Vert\cdot\Vert}(z,\partial\Omega)}\le \sup_{t \in [0,1]}j(t;x,\gamma,\Omega)=\frac{\ell_{\|\cdot\|}(\gamma([0,s]))}{d_{\Vert\cdot\Vert}(\gamma(s),\partial\Omega)} $$
		by \eqref{sp2} and $d_{\|\cdot\|}(\gamma,\partial \Omega)=\inf_{w\in\gamma}d_{\|\cdot\|}(w,\partial \Omega)$, then we have
		\begin{equation}\label{sp7.1}
			\frac{\ell_{\|\cdot\|}(\hat \gamma[\hat{x},z])}{d_{\Vert\cdot\Vert}(z,\partial\Omega)}
			-\frac{\ell_{\|\cdot\|}(\gamma([0,s]))}{d_{\Vert\cdot\Vert}(\gamma(s),\partial\Omega)}\le \frac{\ell_{\|\cdot\|}(\gamma[x,z])+ \|x-\hat{x}\|}{d_{\Vert\cdot\Vert}(z,\partial\Omega)}
			-\frac{\ell_{\|\cdot\|}(\gamma([0,s]))}{d_{\Vert\cdot\Vert}(\gamma(s),\partial\Omega)}\le \frac{\Vert x-\hat{x}\Vert}{d_{\Vert\cdot\Vert}(\gamma,\partial\Omega)}.
		\end{equation}
		Secondly, suppose that $z\in L_{y,\hat{y}}$. Then as \eqref{gamma line} yields
		$$\ell_{\|\cdot\|}(\gamma[\hat x,z]) \le \Vert z-y\Vert + \ell_{\|\cdot\|}(\gamma[x,y])+ \Vert x-\hat{x}\Vert\le \Vert \hat y-y\Vert + \ell_{\|\cdot\|}(\gamma[x,y])+ \Vert x-\hat{x}\Vert,$$
		and the triangle inequality yields
		$$d_{\Vert\cdot\Vert}(z,\partial\Omega)\ge d_{\Vert \cdot\Vert}(y,\partial\Omega)-\Vert y-z\Vert\ge d_{\Vert \cdot\Vert}(y,\partial\Omega)-\Vert y-\hat{y}\Vert,$$
		it follows from \eqref{sp8} that
		\begin{align}
			&\frac{\ell_{\|\cdot\|}(\hat \gamma[\hat x,z])}{d_{\Vert\cdot\Vert}(z,\partial\Omega)}
			-\frac{\ell_{\|\cdot\|}(\gamma([0,s]))}{d_{\Vert\cdot\Vert}(\gamma(s),\partial\Omega)}\nonumber\\
			\le& \frac{\Vert x-\hat{x}\Vert +\Vert\hat{y} -y\Vert }{d_{\Vert \cdot\Vert}(y,\partial\Omega)-\Vert y-\hat{y}\Vert} + \frac{ \ell_{\|\cdot\|}(\gamma[x,y]) }{d_{\Vert \cdot\Vert}(y,\partial\Omega)-\Vert y-\hat{y}\Vert}-\frac{\ell_{\|\cdot\|}(\gamma([0,s]))}{d_{\Vert\cdot\Vert}(\gamma(s),\partial\Omega)}\nonumber\\
			\le &  \frac{\Vert x-\hat{x}\Vert +C_{\|\cdot\|}\Vert y-\hat{y}\Vert }{d_{\Vert\cdot\Vert}(y,\partial\Omega)-\Vert y-\hat{y}\Vert}+  \frac{\Vert y-\hat{y}\Vert}{d_{\Vert\cdot\Vert}(y,\partial\Omega)(d_{\Vert\cdot\Vert}(y,\partial\Omega)-\Vert \hat{y}-y\Vert)} \ell_{\|\cdot\|}(\gamma[x,y])\nonumber\\
			\le &  \frac{\Vert x-\hat{x}\Vert d_{\Vert\cdot\Vert}(y,\partial\Omega)+C_{\|\cdot\|}\Vert y-\hat{y}\Vert d_{\Vert\cdot\Vert}(y,\partial\Omega)+  \Vert y-\hat y\Vert \ell_{\|\cdot\|}(\gamma[x,y]) }{d_{\Vert\cdot\Vert}(y,\partial\Omega)\left(d_{\Vert\cdot\Vert}(y,\partial\Omega)-\Vert y-\hat{y}\Vert\right)}\nonumber\\
			\le& \frac{2C_{\|\cdot\|}( d_{\Vert\cdot\Vert}(y,\partial\Omega)+\ell_{\|\cdot\|}(\gamma[x,y]))}{\left(d_{\Vert\cdot\Vert}(y,\partial\Omega)\right)^2}\left(\Vert x-\hat{x}\Vert +\Vert y-\hat{y}\Vert \right)\le \frac{C(n, C_{\|\cdot\|},\, J) }{ d_{\Vert\cdot\Vert}(y,\partial\Omega)}\left(\Vert x-\hat{x}\Vert +\Vert y-\hat{y}\Vert \right).  
			\label{sp7}   
		\end{align}
		The last case is when $z\in L_{\hat{x},x}$, $ \Vert x- z\Vert \le  \Vert x-\hat{x}\Vert$
		and then
		$$d_{\Vert\cdot\Vert}(z,\partial\Omega)\ge d_{\Vert\cdot\Vert}(x,\partial\Omega)-\Vert x-z\Vert \ge d_{\Vert\cdot\Vert}(x,\partial\Omega)-\Vert x-\hat{x}\Vert. $$
		Thus 	we obtain that
		\begin{equation}\label{sp7.2}
			\frac{\ell_{\|\cdot\|}(\hat \gamma[\hat x,z])}{d_{\Vert\cdot\Vert}(z,\partial\Omega)}
			-\frac{\ell_{\|\cdot\|}(\gamma([0,s]))}{d_{\Vert\cdot\Vert}(\gamma(s),\partial\Omega)} \le  \frac{\Vert x-\hat{x}\Vert}{d_{\Vert\cdot\Vert}(x,\partial\Omega)-\Vert x-\hat{x}\Vert}-\frac{\ell_{\|\cdot\|}(\gamma([0,s]))}{d_{\Vert\cdot\Vert}(\gamma(s),\partial\Omega)}\le \frac{2\Vert x-\hat{x}\Vert}{d_{\Vert\cdot\Vert}(x,\partial\Omega) }.
		\end{equation}

		All in all, we conclude from \eqref{sp7.1},\eqref{sp7} and \eqref{sp7.2} that, for any $t\in [0,1]$,
		\begin{multline}\label{sp3}
			\frac{\ell_{\|\cdot\|}(\hat \gamma([0,t]))}{d_{\Vert\cdot\Vert}(\hat\gamma(t),\partial\Omega)}
			-\frac{\ell_{\|\cdot\|}(\gamma([0,s]))}{d_{\Vert\cdot\Vert}(\gamma(s),\partial\Omega)}\\
			\le
			\max \left\{
			\frac{\Vert x-\hat{x}\Vert}{d_{\Vert\cdot\Vert}(\gamma,\partial\Omega)},\,
			\frac{C(n,C_{\|\cdot\|},\, J) }{ d_{\Vert\cdot\Vert}(y,\partial\Omega)}\left(\Vert x-\hat{x}\Vert +\Vert y-\hat{y}\Vert \right) ,\,
			\frac{2\Vert x-\hat{x}\Vert}{d_{\Vert\cdot\Vert}(x,\partial\Omega) } \right\}.
		\end{multline}
		As a result,  we conclude from \eqref{sp1.5} that
		\begin{align} \label{upper}
			&J(\hat{x},\Omega;\hat{y})-J(x,\Omega;y)\nonumber\\
			\le&\sup_{t\in 
 [0,1]}\frac{\ell_{\|\cdot\|}(\hat \gamma([0,t]))}{d_{\Vert\cdot\Vert}(\hat \gamma(t),\partial\Omega)}-\frac{\ell_{\|\cdot\|}(\gamma([0,s]))}{d_{\Vert\cdot\Vert}(\gamma(s),\partial\Omega)}\nonumber\\
			\le &\max \left\{
			\frac{\Vert x-\hat{x}\Vert}{d_{\Vert\cdot\Vert}(\gamma,\partial\Omega)},\,
			\frac{C(n,C_{\|\cdot\|},\, J)  }{ d_{\Vert\cdot\Vert}(y,\partial\Omega)}\left(\Vert x-\hat{x}\Vert +\Vert y-\hat{y}\Vert \right) ,\,
			\frac{2\Vert x-\hat{x}\Vert}{d_{\Vert\cdot\Vert}(x,\partial\Omega) } \right\}\nonumber\\
			\le &  \frac{C(n,C_{\|\cdot\|},\, J) }{d_{\Vert\cdot\Vert}(\gamma,\partial\Omega)}\left(\Vert x-\hat{x}\Vert +\Vert y-\hat{y}\Vert \right)\le \frac{C(n,C_{\|\cdot\|},\, J) }{d_{\Vert\cdot\Vert}(x,\partial\Omega)}\left(\Vert x-\hat{x}\Vert +\Vert y-\hat{y}\Vert \right).
		\end{align}

		\noindent {\bf Step 2: Estimate $J(x,\Omega;y)-J(\hat{x},\Omega;\hat{y})$ from above. }
		Similarily, we repeat the argument and gain the following estimate:
		\begin{align}\label{lower es}
			&J(x,\Omega;y)-J(\hat{x},\Omega;\hat{y})\le \frac{C(n,C_{\|\cdot\|},\, J) }{d_{\Vert\cdot\Vert}(x,\partial\Omega)}\left(\Vert x-\hat{x}\Vert +\Vert y-\hat{y}\Vert \right),
		\end{align}
		when $\Vert x-\hat{x}\Vert+\Vert y-\hat{y}\Vert<\delta$ for a constant $\delta=\delta(x,y,C_{\|\cdot\|})$   satisfying 
		$$0<\delta\le\frac{1}{2}\min\{d_{\Vert \cdot\Vert}(x,\partial\Omega),d_{\Vert \cdot\Vert}(y,\partial\Omega)\}.$$ 
		Detailed proof of \eqref{lower es} is included in the Appendix~\ref{lower bound}.
		
		\noindent {\bf Step 3: Conclusion. }
		Combining \eqref{upper} and \eqref{lower es},
		we get that $J(\cdot,\Omega,\cdot)$ is continuous and 
		\begin{equation}\label{continui}
			\vert J(x,\Omega;y)-J(\hat{x},\Omega;\hat{y})\vert \le \frac{C(n,C_{\|\cdot\|},\, J) }{d_{\Vert\cdot\Vert}(x,\partial\Omega)}\left(\Vert x-\hat{x}\Vert +\Vert y-\hat{y}\Vert \right),
		\end{equation}
		when  $\Vert x-\hat{x}\Vert+\Vert y-\hat{y}\Vert<\delta$.
		Thus
		for any $(x,y)\in\Omega\times\Omega$, by letting  
		$$U_{x,y}:=\left\{(a,b)\in\Omega\times\Omega:\Vert a-x\Vert+\Vert b-y\Vert<\frac{1}{16C_{\|\cdot\|}}\delta\right\}, $$
		the estimate \eqref{continui} yields that whenever  $(x_1,y_1),(x_2,y_2)\in U_{x,y}$,  
		\begin{equation}\label{lip} 
			|J(x_1,\Omega;y_1)-J(x_2,\Omega;y_2)|\le \frac{C_{x,y}}{d_{x,y}}\left(\Vert x_1-x_2\Vert+\Vert y_1-y_2\Vert\right),
		\end{equation}
		where 
		$$C_{x,y}=\max_{(a,b)\in \overline{U}_{x,y}}C(n,C_{\|\cdot\|},\, J(a,\Omega;b))<\infty$$
		by the John assumption on $\Omega$,  
		and $$d_{x,y}=\min_{(a,b)\in \overline{U}_{x,y}}d_{\Vert\cdot\Vert}(a,\partial \Omega).$$
		From \eqref{lip} we finally conclude that  $J(\cdot,\Omega;\cdot)$ is locally Lipschitz continuous.


	\end{proof}
	
	Recall that
	$$J(\Omega;x_0):=\sup_{x\in\Omega}J(x,\Omega;x_0). $$

	\begin{lem}\label{continuous 2}
		Let $\Omega\subset\mathbb{R}^n$ be a bounded John domain. Then $J(\Omega;\cdot):\Omega\to [1,+\infty)$ is a lower-semicontinuous function, such that 
		\begin{enumerate}
			\item[(i)] For $y\in \Omega$,  $r_{\Omega}:=\max_{z\in\Omega}\{d_{\Vert \cdot\Vert}(z,\partial\Omega)\}$, we have  
			$$J(\Omega;y)\ge \frac{r_{\Omega}-d_{\Vert \cdot\Vert}(y,\partial\Omega)}{C_{\|\cdot\|} d_{\Vert \cdot\Vert}(y,\partial\Omega)};$$ 
			\item[(ii)] Let $x_{\Omega}\in\Omega$ be a point with $d_{\Vert\cdot\Vert}(x_{\Omega},\partial\Omega)=r_{\Omega}$. Then  $J(\Omega;\cdot)$ attains its infimum in 
			$$\{x\in\Omega: d_{\Vert\cdot\Vert}(x,\partial\Omega)\ge r_0\},$$ where $r_0:=\frac{r_{\Omega}}{1+2C_{\|\cdot\|}J(\Omega;x_{\Omega})}>0.$
		\end{enumerate}
		
	\end{lem}

	\begin{proof}
		Observe that for any $x\in\Omega$, $J(x,\Omega;\cdot)$ is a continuous function in $\Omega$ by Lemma~\ref{continuous 1}. Then we find that $J(\Omega;\cdot)$  is a lower-semicontinuous function in $\Omega$ since $$J(\Omega;\cdot)=\sup_{x\in\Omega}J(x,\Omega;\cdot).$$ 
		
		Now we prove (i). 
		Let $x_{\Omega}\in\Omega$ be a point satisfying 
		$$d_{\Vert \cdot\Vert}(x_{\Omega},\partial\Omega)=\max_{y\in\Omega}\{d_{\Vert \cdot\Vert}(y,\partial\Omega)\}=:r_{\Omega}.$$ 
		For any $y\in \Omega$, combining  the definition of $J(x_{\Omega},\Omega;y)$, Lemma~\ref{J eqiup} and the triangle inequality, there exists a  rectifiable curve $\gamma\subset\Omega$ joining $x_{\Omega}$ to  $y$, such that 
		
		$$J(x_{\Omega},\Omega;y)=\sup_{a\in \gamma}\frac{\ell_{\|\cdot\|}( \gamma[x_{\Omega},a])}{d_{\Vert \cdot\Vert}(a,\partial\Omega)}\ge \frac{\ell_{\|\cdot\|}( \gamma[x_{\Omega},y])}{d_{\Vert \cdot\Vert}(y,\partial\Omega)}\ge\frac{\ell_{\|\cdot\|}( \gamma[y,x_{\Omega}])}{C_{\|\cdot\|}d_{\Vert \cdot\Vert}(y,\partial\Omega)}\ge\frac{r_{\Omega}-d_{\Vert \cdot\Vert}(y,\partial\Omega)}{C_{\|\cdot\|}d_{\Vert \cdot\Vert}(y,\partial\Omega)}.$$
		Then we have 
		\begin{equation}\label{sp9}
			J(\Omega;y)\ge J(x_{\Omega},\Omega;y)\ge \frac{r_{\Omega}-d_{\Vert \cdot\Vert}(y,\partial\Omega)}{C_{\|\cdot\|}d_{\Vert \cdot\Vert}(y,\partial\Omega)}.
		\end{equation}
		
		Now we proceed to (ii). 	
		Recall that   $J(\Omega;x_{\Omega})<+\infty$ by Lemma~\ref{4c^2}. We define
		$$r_0:=\frac{r_{\Omega}}{1+2C_{\|\cdot\|}J(\Omega;x_{\Omega})} \qquad \text{and}\qquad \Omega_{r_0}:=\{x\in\Omega: d_{\Vert \cdot\Vert}(x,\partial\Omega)> r_0\}$$
		so that for any $0<r\le r_0$
		$$\frac{r_{\Omega}-r}{C_{\|\cdot\|}r}\ge \frac{r_{\Omega}-r_0}{C_{\|\cdot\|}r_0}=2J(\Omega;x_{\Omega}). $$
		Then since $x_{\Omega}\in\Omega_{r_0}$, for any $z\in \Omega\setminus\Omega_{r_0}$, we conclude from \eqref{sp9} that 
		$$J(\Omega;z)\ge 2J(\Omega;x_{\Omega}) > J(\Omega;x_{\Omega})\ge \inf_{x\in\overline{\Omega}_{r_0}}J(\Omega;x).$$ 
		Then the above estimate yields that  $\inf_{x\in\Omega}J(\Omega;x)=\inf_{x\in\overline{\Omega}_{r_0}}J(\Omega;x)$.
		
		Notice that $\overline{\Omega}_{r_0}$ is a compact set and $J(\Omega;\cdot)$ is a lower-semicontinuous function in $\Omega$. As a consequence,    
		there  exist a point $b\in\overline{\Omega}_{r_0}$ such that $$J(\Omega;b)=\inf_{x\in\overline{\Omega}_{r_0}}J(\Omega;x)=\inf_{x\in\Omega}J(\Omega;x).$$	
	\end{proof}

	%

	We further need an auxiliary lemma regarding Hausdorff convergence.

	\begin{lem}\label{keep distance}
		Suppose that $\{K_j\}_{j\in\mathbb{N}}$ is a sequence of compact sets converging to a compact set $K$ in the Hausdorff metric and the interior of $K$ is denoted as $\Omega$. Assume further that 
		$$ \inf_{j\in\mathbb{N}}\max_{x\in K_j}d_{\Vert \cdot\Vert}(x,\partial K_j)\ge r_0.$$
		Then for any  $r\in (0,r_0]$ and any converging sequence $\{x_j\}
		_{j\in\mathbb{N}}$  satisfying $x_j\in K_j$ and 
		$$d_{\Vert \cdot\Vert}(x_j,\partial K_j)\ge r, $$
		the limit $x:=\lim_{j\to \infty} x_j$ satisfies
		\begin{equation}\label{keep dis}
			x\in \Omega \text{ \  and \ } d_{\Vert \cdot\Vert}(x,\partial\Omega)\ge r.
		\end{equation}
	\end{lem}

	\begin{proof} 
		
		As  $K_j$ converge to $K$ in the Hausdorff metric by our assumption,  we claim that  $K$ can be explicitly represented as 
		\begin{equation}\label{Hausdorff limit}
			K=\bigcap_{m=1}^{+\infty}Cl\Big(\bigcup_{j=m}^{+\infty}K_j\Big).
		\end{equation}
		This conclusion can be found in 
		\cite[Exercise 7.3.4]{BBI2022}.

		Now by  \eqref{Hausdorff limit} and the convergence of $x_j$, we have
		\begin{equation}\label{sp14}
			\{x\}=\bigcap_{m=1}^{+\infty}Cl\Big({\bigcup_{j=m}^{+\infty}\{x_j\}}\Big)\subset\bigcap_{m=1}^{+\infty}Cl\Big({\bigcup_{j=m}^{+\infty}K_j}\Big)= K.
		\end{equation}
		Choosing $\ez>0$ sufficiently small and for any $r\in (0,r_0]$,  there is $j_0\in \mathbb{N}$, such that 
		$$\Vert x_j-x\Vert<\ez$$
		for any $ j\ge j_0$. Thus, we get 
		$$d_{\Vert \cdot\Vert}(x,K_j^c)\ge d_{\Vert \cdot\Vert}(x_j,K_j^c)-\Vert x_j-x\Vert\ge r-\ez \qquad\forall j\ge j_0.$$
		The estimate above yields that 
		$$d_{\Vert \cdot\Vert}\Big(x,\Big(Cl\big({\bigcup_{j=m}^{+\infty}K_j}\big)\Big)^c\Big)\ge d_{\Vert \cdot\Vert}\Big(x,\big(\bigcup_{j=m}^{+\infty}K_j\big)^c\Big)\ge r-\ez
		\qquad  \forall m\ge j_0.$$
		
		Let $m\to+\infty $, then we have
		$$d_{\Vert \cdot\Vert}(x,K^c)\ge r-\ez.$$
		Further let $\ez\to 0$, from the above estimate and  (\ref{sp14}) we get (\ref{keep dis}). 
	\end{proof}
	
	Now we are ready to show Theorem \ref{low}.  
	
	\begin{proof}[Proof of Theorem \ref{low}]
		Assume that 	$$J:=\liminf_{j\to+\infty}John(\Omega_j)\le J_0<\infty.$$
		Let 
		$$\Omega_{j,r}:=\{x\in \Omega_j: d_{\Vert\cdot\Vert}(x,\partial \Omega_j)\ge r\} \qquad \text{and}\qquad \Omega_{r}:=\{x\in \Omega: d_{\Vert\cdot\Vert}(x,\partial \Omega)\ge r\}$$
		for some $r>0$ to be determined. 
		Further let $\{\Omega_j\}_{j\in\mathbb{N}^+}$ be a minimizing sequence and $x_{\Omega_j}\in \Omega_j$ be a point  satisfying 
		$$d_{\Vert \cdot\Vert}(x_{\Omega_j},\partial\Omega_j)=\max_{x\in\Omega_j}d_{\Vert \cdot\Vert}(x,\partial\Omega_j)=:r_{\Omega_j}.$$ 
		
		On the other hand, by Lemma~\ref{continuous 2},  for each $i\in\mathbb{N}$ there exists a (center) point $x_{j}\in\Omega_{j,r}$, such that 
		$$
		J(\Omega_j;x_{j})=John(\Omega_j).
		$$
		We remark that $x_j$ might not be $x_{\Omega_j}$. Nevertheless, by Lemma \ref{4c^2} we have
		\begin{equation}\label{sp10}
			J(\Omega_j;x_{\Omega_j})\le C(n,C_{\|\cdot\|},J_0).
		\end{equation}

		In addition, since $(\mathcal{C}^n,d_H)$ is complete and bounded subsets are precompact,  up to passing to a subsequence,   $\{\overline{\Omega}_i\}_{i\in\mathbb{N}}$ converges in the Hausdorff metric to a compact set $A$. We set the interior of $A$ as $\Omega$. 
		
		\noindent{\bf Step 1: $r_{\Omega_j}$ is uniformly bounded away from $0$.} To this end, for each $j\in\mathbb{N}^+$, the definition of $J(\Omega_{j};x_{\Omega_{j}})$ and Lemma \ref{J eqiup} tell that, for any $x\in\Omega_{j}\setminus\{x_{\Omega_{j}}\}$, there exists a rectifiable curve $\beta_{j}$ joining $x$ to $x_{\Omega_{j}}$, such that 
		$$\sup_{a\in 
 [0,1]}\frac{\ell_{\|\cdot\|}(\beta_j([0,a]))}{d_{\Vert \cdot\Vert}(\beta_j(a),\partial\Omega_{j})}=J(x,\Omega_{j};x_{\Omega_{j}})\le J(\Omega_{j};x_{\Omega_{
				j}})$$
		and thus by \eqref{sp10}
		\begin{multline*}
			\ell_{\|\cdot\|}(\beta_j[x,x_{\Omega_j}])\le \sup_{a\in [0,1]}\frac{\ell_{\|\cdot\|}(\beta_j([0,a]))}{d_{\Vert \cdot\Vert}(\beta_j(a),\partial\Omega_{j})}d_{\Vert \cdot\Vert}(x_{\Omega_{j}},\partial\Omega_{j})\\ 
			\le J(\Omega_{j};x_{\Omega_{j}})d_{\Vert \cdot\Vert}(x_{\Omega_{j}},\partial\Omega_{j})\le C(n,C_{\|\cdot\|},\, J_0) r_{\Omega_{j}}.
		\end{multline*}
		This yields that 
		$ \Omega_{j}\subset B_{\Vert \cdot\Vert}(x_{\Omega_{j}}, C(n,C_{\|\cdot\|},\, J_0)r_{\Omega_{j}}),$
		from which  we conclude
		$$c_0\vert B_{\Vert \cdot\Vert}(0,1)\vert\le \vert \Omega_{j}\vert\le \vert B_{\Vert \cdot\Vert}(0,1)\vert( C(n,C_{\|\cdot\|},\, J_0) r_{\Omega_{j}})^n.$$
		As a result, we conclude that
		\begin{equation}\label{infimum distance keep}
			r_{\Omega_j}\ge c
		\end{equation}
		for some $c=c(n,\, C_{\|\cdot\|}, \,J_0,\,c_0)>0.$
		
		\noindent{\bf Step 2: $x_j$  is uniformly away from the boundary.} 	
		Up to further extracting a subsequence, we may assume 
		$$x_\Omega:=\lim_{j\to+\infty}x_{\Omega_j}\in \mathbb R^n. $$
		Recalling Lemma \ref{keep distance} and  (\ref{infimum distance keep}), it follows  that 
		$$\max_{x\in\Omega}d_{\Vert \cdot\Vert}(x,\partial\Omega)\ge d_{\Vert \cdot\Vert}(x_\Omega,\partial\Omega)\ge c>0.$$
		In addition, we can choose $r>0$ so that
		$$ r\le \inf_{j\in\mathbb{N}}\frac{r_{\Omega_j}}{1+2C_{\|\cdot\|}J(\Omega;x_{\Omega_j})},$$ 
		where its existence is ensured     by \eqref{infimum distance keep}, \eqref{sp10} and Lemma~\ref{continuous 2}. 
		
		Recall that $x_j\in\Omega_{j,r}$. Then up to further passing to a subsequence,  we may assume that the limit $x $ of $\{x_j\}_{j\in\mathbb{N}}$ exists, and Lemma~\ref{keep distance}  implies that $x\in\Omega_{r}$ .

		\noindent{\bf Step 3: Lower semicontinuity of  $John(\Omega_j)$.} 	
		Let $J_i:=John(\Omega_j)$.  Note that for any $y\in\Omega$, Hausdorff convergence yields the existence of a sequence $y_j\in\overline{\Omega}_j$ such that $\lim_{j\to+\infty} y_j=y$. Combining the definition of $John(\Omega_j)$, Lemma~\ref{J eqiup} and Lemma~\ref{continuous 2}, for each $j\in\mathbb N$ we obtain a rectifiable curve $\gamma_j\subset\Omega_j$ joining $y_j$ to $x_j$ such that  the corresponding $J_i$-carrot $car(\gamma_j,J_j)$ satisfies 
		$$car(\gamma_j,J_j)\subset\Omega_j,$$
		which, due to the fact that $\Omega_j$ is uniformly bounded and  $J_j\le J_0$ for each $j\in\mathbb N^+$, yields that $\ell_{\|\cdot\|}(\gamma_j[y_j,x_j])\le J_0 \max_{j\in\mathbb N^+}r_{\Omega_j}.$ 
		Then by Lemma \ref{carrot convergent}, there exists a rectifiable curve  $\gamma\subset \Omega$  joining $y$ to $x$ so that the $J$-carrot $car(\gamma,J)$  (as a Euclidean open set) satisfies
		\begin{equation}\label{sp16}  car(\gamma,J)\subset\bigcap_{m=1}^{+\infty}\Big({\bigcup_{j=m}^{+\infty}car(\gamma_j,J_j)}\Big)\subset\bigcap_{m=1}^{+\infty}Cl\Big({\bigcup_{j=m}^{+\infty}\overline{\Omega}_j}\Big)= A.
		\end{equation}
		As $\Omega$ is the interior of $A$, then from (\ref{sp16}) we have  $car(\gamma,J)\subset \Omega$,  which implies that 
		\begin{equation}\label{john at y}
			J(y,\Omega;x )\le J.
		\end{equation}
		To conclude, each $y\in \Omega$ can be joint from $x $ by a rectifiable curve inside $\Omega$, which implies that $\Omega$ is connected. Furthermore, the arbitrariness of $y$ in \eqref{john at y} yields
		$$John(\Omega)\le J(\Omega;x )\le J= \liminf_{j\to+\infty}John(\Omega_j).$$
		We complete the second part of the proof. 
	\end{proof}

	\section{John component of unbounded carrot John domain}\label{unbdd carrot John}

	In this section, we consider the John domain defined via the standard Euclidean norm $|\cdot|$. 
	
	The proof of  Theorem \ref{John component} is rather technical since most of the sets defined in question are open. 
	To obtain the sets $W_{j,\infty}$ as mentioned in  Theorem \ref{John component},  we initially decompose  $B_R\setminus K$  into  at most $C(n,\, J)$-many sets $\{V_{j,\,R}\}$. Subsequently, for each  $y\in V_{j,R}$,   we create a bounded $J'$-carrot John domain $\Omega_{j,R,y}$ (refer to Proposition~\ref{omega cons}). In the proof of Theorem~\ref{John component}, we choose sequences of points $x_{j,r}\in V_{j,R}$ along with the corresponding John curves $\gamma_{x_{j,r}}$ extending from $x_{j,r}$ towards $\infty$, where $r$ is a positive number with 
	$r\le R$.   This selection ensures that we can obtain sets
	$$W_{j,R}=\Omega_{j,\,R,x_{j,r}}\subset B_{C'R}, \quad C'=C'(n,\,J).$$
	In particular, by eventually choosing $x_j\in \mathbb R^n$ suitably, 
	$$W_{j,\infty}:=\bigcup_{R>|x_j|} W_{j,R}\subset \mathbb{R}^n\setminus K$$
	fulfills the condition that for every pair of distinct points $z,w\in W_{j,\infty}$, there exists a point $a\in \gamma_{x_j}$ to which both $z$ and $w$ can be connected by $\hat \gamma_z$ and $\hat \gamma_w$, respectively. Moreover, these connecting curves satisfy the properties \eqref{Boman chain} and \eqref{cigar ball}.

	Prior to identifying the desired bounded John domain $\Omega_{j,R,y}$, we rely on the following proposition. While this technique has been commonly employed in previous manuscripts, such as \cite{NV1991}, it has not been explicitly formulated, to the best of our knowledge, in the context of our present work.

	\begin{prop}\label{length car pro}
		Let $J\ge 1$. Assume that $\gamma \subset\mathbb{R}^n$ is a locally rectifiable curve joining $x$ to $y$, where $x\in \mathbb{R}^n$ and $y\in \dot{\mathbb{R}}^n$ ($y$ may be $\infty$). Then
		$ car(\gamma,J)$ is a $J$-carrot John domain. 
		
		To be more specific,  for any $z\in car(\gamma,J)$, we can find a rectifiable curve $\gamma_z$ joining $z$ to $y$, such that for some $\eta\in \gamma$, we have
		$$\gamma[\eta,y]=\gamma_z[\eta,y]$$
		and for each $a\in \gamma[\eta,y]\setminus \{\infty\}$,  
		\begin{equation}\label{length}
			\ell(\gamma_z[z,a])\le  \ell(\gamma[x,a]),\quad car(\gamma_z,J)\subset car(\gamma,J).
		\end{equation}
	\end{prop}
	\begin{proof}
		For any $z\in car (\gamma,J)$, the definition of $car (\gamma,J)$ yields a ball 
		$$B(\eta,\ell(\gamma[x,\eta])/J)\subset car (\gamma,J)$$
		for some points $\eta\in \gamma\setminus \{x\}$ so that $z\in B(\eta,\ell(\gamma[x,\eta])/J)$.
		
		Let $L_{z,\eta}$ be the line segment joining $z$ to $\eta$ and then $\gamma_z:=L_{z,\eta}\cup \gamma[\eta,y]$ is a locally  rectifiable curve joining $z$ to $y$. When $a\in L_{z,\eta}$,  
		\begin{equation}\label{a41}
			\ell(\gamma_z[z,a])\le d\left(a,\partial B(\eta,\ell(\gamma[x,\eta])/J)\right)\le \ell(\gamma[x,\eta])/J.
		\end{equation}
		When $a\in \gamma[\eta,y]$, by applying \eqref{a41} with $a=\eta$ there, 		we have 
		\begin{align*} 
			\ell(\gamma_z[z,a])&\le \ell(\gamma_z[z,\eta])+\ell(\gamma_z[\eta,a])\le \frac{\ell(\gamma[x,\eta])}{J}+\ell(\gamma[\eta,a])\nonumber\\
			& \le \ell(\gamma[x,\eta])+\ell(\gamma[\eta,a])=\ell(\gamma[x,a]).
		\end{align*}
		To conclude, we obtain that
		$$\ell(\gamma_z[z,a])\le  \ell(\gamma[x,a]),$$
		which is the first formula of \eqref{length}. 
		The second one follows directly from our construction of $car(\gamma_z, J)$ and $car(\gamma, J)$, and we conclude the lemma. 
	\end{proof}
	
	\subsection{A decomposition $V_{j,\,R}$ of $B_R\setminus K$}
	
	Now  for any $x\in\mathbb{R}^n\setminus K$, 
	we choose a John curve $\gamma_x\subset \mathbb{R}^n\setminus K$ joining $x$ towards $\infty$ with $car(\gamma,J)\subset \mathbb R^n\setminus K$.
	Although there could be many choices of curves for $x\in\mathbb{R}^n\setminus K$, we just  choose one of them. Let $\Gamma=\{\gamma_x\}_{x\in \mathbb{R}^n\setminus K }$ be the collection of these chosen curves. In what follows, for any points $x\in \mathbb{R}^n\setminus K$, $\gamma_x$ always refers to this particular choice of John curve.
	
	Note that for any $R> 0$, we have  $B_{R}\cap K \neq \emptyset$ as $0\in K$. Our first step is to decompose $B_R\setminus K$ into finitely many subsets $V_{j,R}$ so that, there exists a collection $\mathcal{B}_{j,R}$ of at most $C(n,\, J)$-many balls, whose centers are on $\partial B_{3R}$ and whose radii at least $J^{-1}R$, satisfying that, 
	for  any $x\in V_{j,R}$, we can find a ball $B\in\mathcal{B}_{j,R}$ with
	$$\gamma_x\cap B\neq\emptyset.$$

	To this end, observe that for each $x\in B_R\setminus K$ and $\gamma_x\in \Gamma$, there exists a point
	\begin{equation}\label{xR def}
		x_R\in\gamma_x\cap\partial B_{3R}  
	\end{equation}
	so that 
	\begin{equation}\label{lower size}
		2R\le\ell(\gamma[x,x_R])\le Jd(x_R,K).
	\end{equation}
	
	Consider the collection of closed balls 
	\begin{equation}\label{ball collect}
		\{\overline{B}_x\}_{x\in B_R\setminus K}:=\left\{\overline{B}\left(x_R,\frac{d(x_R,K)}{2}\right)\right\}_{x\in B_R\setminus K}.
	\end{equation}
	Then thanks to \eqref{lower size} and $0\in K$, we  obtain that 
	\begin{equation}\label{radius larg}
		\frac{R}{J}\le \frac{d(x_R,K)}{2}\le 2R,
	\end{equation}
	and hence  $B_x\cap B_R=\emptyset$. 
	
	We next let 
	$$A_R:=\bigcup_{x\in \overline{B_R}\setminus K}\{x_R\}$$
	be the collection of the centers of $B_x$'s. By Bescovitch's covering theorem, there exists a subcollection  $\{\overline{B}_i\}_{i\in\mathbb{N}}$ of $\{\overline{B}_x\}_{x\in B_R\setminus K}$ consisting of at most countably many balls, such that 
    \begin{equation}\label{bcover}
        \chi_{A_R}(z)\le \sum_{B_i} \chi_{\overline{B}_i}(z)\le C(n) \qquad  \forall z\in \mathbb{R}^n\setminus K; 
    \end{equation}
	see Figure~\ref{fig:K}.

	\begin{figure}[ht]
		\centering
		\def\svgwidth{\columnwidth}
		\resizebox{0.7\textwidth}{!}{
\begingroup%
  \makeatletter%
  \providecommand\color[2][]{%
    \errmessage{(Inkscape) Color is used for the text in Inkscape, but the package 'color.sty' is not loaded}%
    \renewcommand\color[2][]{}%
  }%
  \providecommand\transparent[1]{%
    \errmessage{(Inkscape) Transparency is used (non-zero) for the text in Inkscape, but the package 'transparent.sty' is not loaded}%
    \renewcommand\transparent[1]{}%
  }%
  \providecommand\rotatebox[2]{#2}%
  \newcommand*\fsize{\dimexpr\f@size pt\relax}%
  \newcommand*\lineheight[1]{\fontsize{\fsize}{#1\fsize}\selectfont}%
  \ifx\svgwidth\undefined%
    \setlength{\unitlength}{209.76377106bp}%
    \ifx\svgscale\undefined%
      \relax%
    \else%
      \setlength{\unitlength}{\unitlength * \real{\svgscale}}%
    \fi%
  \else%
    \setlength{\unitlength}{\svgwidth}%
  \fi%
  \global\let\svgwidth\undefined%
  \global\let\svgscale\undefined%
  \makeatother%
  \begin{picture}(1,0.53631759)%
    \lineheight{1}%
    \setlength\tabcolsep{0pt}%
    \put(0,0){\includegraphics[width=\unitlength,page=1]{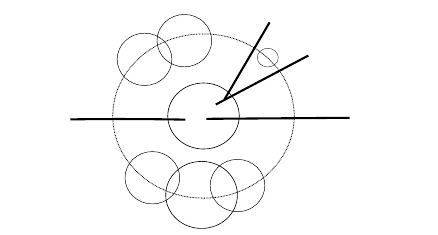}}%
    \put(0.7150901,0.22793499){\makebox(0,0)[lt]{\lineheight{1.25}\smash{\begin{tabular}[t]{l}$K$\end{tabular}}}}%
    \put(0.4710656,0.21452705){\makebox(0,0)[lt]{\lineheight{1.25}\smash{\begin{tabular}[t]{l}$B_R$\end{tabular}}}}%
    \put(0.67129083,0.35665121){\makebox(0,0)[lt]{\lineheight{1.25}\smash{\begin{tabular}[t]{l}$\partial B_{3R}$\end{tabular}}}}%
    \put(0,0){\includegraphics[width=\unitlength,page=2]{K.pdf}}%
  \end{picture}%
\endgroup%
}
		\caption{The set $K$ is the union of black lines. We apply Bescovitch's covering theorem to cover the set $A_R$ with balls centered at $\partial B_{3R}$. }
		\label{fig:K}
	\end{figure}
	
	
	Recall that by \eqref{radius larg}
	$$B_i\subset B_{5R}\setminus \overline{B}_{R}$$
	and
	$|B_{i}|\ge c(n,\,J)R^n. $
	Thus we have at most $C(n,\,J)$-many elements in $\{\overline{B_i}\}$ by \eqref{bcover}. 
	As a result, the union of balls 
	$$\bigcup_{i}\overline{B}_i$$
	has at most $\hat N=\hat{N}(n,\,J)$ components $U_{j,R}$
	for $j\in\{1,\cdots,N_R\}$ and 
	$$N_R\le \hat N=\hat N(n,\,J);$$
	By defining $U_{j,\,R}$ to be empty for $j> N_R$, we may assume that there exist exactly $\hat N$ components $U_{j,R}$, and each $U_{j,R}$ contains at most $\hat{N}$ balls. 
	We write 
\begin{equation}\label{B collection}
    \mathcal {B}_{j,R}=\{B_i: B_i\subset U_{j,R}\}
\end{equation}
	
	Now it follows from our construction, for any $x\in B_R\setminus K$, there exists some $1\le j\le N_R$ so that, $x_R\in \gamma_x$ is covered by a ball in $\mathcal{B}_{j,R}$. Thus, by  defining 
	\begin{equation}\label{VR}
		V_{j,R}:=\{x\in B_R\setminus K:x_R\in D \text{ for some }D\in\mathcal{B}_{j,R}\},
	\end{equation}
	we obtain the desired decomposition of $B_R\setminus K$. The set $V_{j,\,R}$ is defined to be empty if $U_{j,\,R}$ is empty.

	\subsection{Construction of $\Omega_{j,\,R,\,y}$}	
	Given $R>0$ and $j\in\{1,\cdots,N_R\}$, recall the construction of $\mathcal{B}_{j,R}$ and $V_{j,R}$ in the last subsection. Then for each point $y\in V_{j,R}$ we set up a bounded $J'$-carrot John domain $\Omega_{j,R,y}$ with John center  $y_R$, where $J'=J'(n,\,J)$, such that 
	$$V_{j,R}\subset \overline{\Omega}_{j,R,y}\quad \text{and}\quad \Omega_{j,R,y}\subset (\mathbb{R}^n\setminus K)\cap B_{C'R}$$
	where $C'=C'(n,\,J)$; see Figure~\ref{fig:O}. We formulate it as the following proposition. 
	
	\begin{figure}[ht]
		\centering
		\def\svgwidth{\columnwidth}
		
		\resizebox{0.7\textwidth}{!}{
\begingroup%
  \makeatletter%
  \providecommand\color[2][]{%
    \errmessage{(Inkscape) Color is used for the text in Inkscape, but the package 'color.sty' is not loaded}%
    \renewcommand\color[2][]{}%
  }%
  \providecommand\transparent[1]{%
    \errmessage{(Inkscape) Transparency is used (non-zero) for the text in Inkscape, but the package 'transparent.sty' is not loaded}%
    \renewcommand\transparent[1]{}%
  }%
  \providecommand\rotatebox[2]{#2}%
  \newcommand*\fsize{\dimexpr\f@size pt\relax}%
  \newcommand*\lineheight[1]{\fontsize{\fsize}{#1\fsize}\selectfont}%
  \ifx\svgwidth\undefined%
    \setlength{\unitlength}{209.76377106bp}%
    \ifx\svgscale\undefined%
      \relax%
    \else%
      \setlength{\unitlength}{\unitlength * \real{\svgscale}}%
    \fi%
  \else%
    \setlength{\unitlength}{\svgwidth}%
  \fi%
  \global\let\svgwidth\undefined%
  \global\let\svgscale\undefined%
  \makeatother%
  \begin{picture}(1,0.53631759)%
    \lineheight{1}%
    \setlength\tabcolsep{0pt}%
    \put(0,0){\includegraphics[width=\unitlength,page=1]{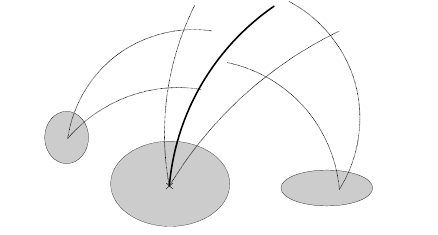}}%
    \put(0.40545392,0.08672564){\makebox(0,0)[lt]{\lineheight{1.25}\smash{\begin{tabular}[t]{l}$y$\end{tabular}}}}%
    \put(0.44240672,0.29099229){\makebox(0,0)[lt]{\lineheight{1.25}\smash{\begin{tabular}[t]{l}$\gamma_y[y,\,y_R]$\end{tabular}}}}%
    \put(0.79756377,0.10930786){\makebox(0,0)[lt]{\lineheight{1.25}\smash{\begin{tabular}[t]{l}$V_{j,\,R}$\end{tabular}}}}%
    \put(0,0){\includegraphics[width=\unitlength,page=2]{O.pdf}}%
  \end{picture}%
\endgroup%
}
		\caption{The set $V_{j,\,R}$ may not necessarily be connected. We connect each point in $V_{j,\,R}$ to the curve $\gamma_y[y,y_R]$ using appropriate curves. Subsequently, we take the union of the carrots surrounding these curves to form $\Omega_{j,\,R,\,y}. $} 
		\label{fig:O}
	\end{figure}
	
	\begin{prop}\label{omega cons}
		For fixed $y\in V_{j,\,R}$ and $1\le j\le \hat N$ with $\hat N=\hat N(n,\,J)$ defined above, the set 	\begin{equation}\label{Omega def}
			\Omega_{j,R,y}:=car(\gamma_y[y,y_R],J)\cup \bigcup_{z\in V_{j,R}}car(\beta_z,J').
		\end{equation}
		is a $J'$-carrot John domain with John center $y_R$, where $J'=J'(n,\, J)$ and  $\beta_z$ is a rectifiable curve joining $z$ to $y_R$ satisfying $\gamma_{z}[z,z_R]\subset \beta_z$; recall that $\gamma_x$ is a chosen curve joining $x$ toward $\infty$.
        
        Moreover, there exists $C_1=C_1(n,\, J)\ge 4$ so that, the curve $\beta_z$ joining $z\in V_{j,R}$ to $y_R$ that is the core of a $J'$-carrot satisfying 
		$$\ell(\beta_z)\le C_1 R$$
		and
		\begin{equation}\label{close omega}
			V_{j,R}\subset \overline{\Omega}_{j,R,y},\quad car(\beta_z,J')\subset\Omega_{j,R,y}\subset (\mathbb{R}^n\setminus K)\cap B_{2C_1R}.
		\end{equation}
	\end{prop}
	\begin{proof}
		Suppose that $ V_{j,R}$ is non-empty, and fix $y\in V_{j,R}$. Then the corresponding point $y_R\in \gamma_y\cap \partial B_{3R}$ is covered by some ball $D_1\in \mathcal{B}_{j,R}$ according to \eqref{B collection} and \eqref{VR}. Then we join the center $\hat{x}_1$ of $D_1$ to $y_R$ by a line segment $L_{\hat{x}_1,y_R}\subset D_1$. 
		
		Now for any $z\in V_{j,R}$, we claim that there exists  a rectifiable curve $\beta_z\subset \mathbb R^n\setminus K$ as the core of a $J'$-carrot joining $z$ to $y_R$, such that $\gamma_{z}[z,z_R]\subset \beta_z$ and
		\begin{equation}\label{base point ball}
			car(\beta_{z},J')\subset \mathbb{R}^n\setminus K.
		\end{equation}
		
		Indeed, the point $z_R\in \gamma_z\cap \partial B_{3R}$ is also covered by another ball $D_2\in\mathcal{B}_{j,R}$ as $z\in V_{j,R}$. Likewise, we join $z_R$  to the center $\hat{x}_2$ of $D_2$ by the line segment $L_{z_R,\hat{x}_2}\subset D_2$.

		Recall that $U_{j,R}$ is connected and consists of at most $\hat{N}$-many  balls from $\mathcal{B}_{j,R}$, where $\hat N = \hat N(n,\,J)$.  
		This implies that $\hat{x}_1$ and $\hat{x}_2$ can be joined by a union of at most $\hat N$-many line segments with the endpoints being the centers of balls in  $\mathcal{B}_{j,R}$. Therefore, combining with $L_{z_R,\hat{x}_2}$ and $L_{\hat{x}_1,y_R}$, we can join $z_R$ to $y_R$ by a polyline $\gamma_{z_R,y_R}$.  
		
		We show that 
		$$\beta_z:=\gamma_z[z,z_R]\cup \gamma_{z_R,y_R}$$
		is the desired John curve. 
		To this end, we estimate the length of $\beta_z$ and the distance $d(\eta,K)$ for any $\eta\in \beta_z$, respectively. 
        
        We start with the estimate on the length of $\beta_z$. Thanks to \eqref{radius larg} and \eqref{B collection}, for any pair of intersecting balls $D,D'\in \mathcal{B}_{j,R}$,  the line segments $L$ joining the center of $D$ with radius $r$ to the center of $D'$ with radius $r'$ satisfies 
		\begin{equation}\label{length estimate 2}
			L\subset D\cup D'\quad \text{and}\quad \ell(L)\le r+r'\le 4R.
		\end{equation}
		In particular, \eqref{radius larg} together with the facts that $L_{z_R,\hat{x}_2}\subset D_2$ and that $L_{\hat{x}_1,y_R}\subset D_1$ also yields
		$\ell(L_{z_R,\hat{x}_2})\le 2R$, $\ell(L_{\hat{x}_1,y_R})\le 2R$. 
		Therefore employing   \eqref{length estimate 2} and \eqref{lower size},  the construction of $\beta_z$ tells
		\begin{align}\label{length beta}
			\ell(\beta_z)&\le \ell(\gamma_z[z,z_R])+\ell(\gamma_{z_R,y_R}) \nonumber\\
			&\le Jd(z_R,\,K)+\ell(L_{z_R,\hat{x}_2})+\ell(L_{\hat{x}_1,y_R})+4\hat{N}(n,\,J)R \nonumber\\
			&\le C(n,\,J) R=:C_1R;
		\end{align}
		we may assume that	$C_1\ge 4$. 
		This gives the first part of the proposition.
		
		Towards \eqref{close omega}, for any $\eta\in \beta_z$, we need to estimate the distance $d(\eta, K)$ from above.
        First of all, note that when $\eta\in\gamma_{z_R,y_R}$, there exists some ball $D_{\eta}\in \mathcal{B}_{j,R}$ containing $\eta$. Then combining \eqref{lower size},\eqref{ball collect} and \eqref{radius larg}, we get 
		\begin{equation}\label{small ball}
			d(\eta,K)\ge d(D_\eta,\,K)\ge  \frac{R}{J}.
		\end{equation}	
		Let 
	\begin{equation}\label{carrot construct 2}
			J':=C_1 J.
		\end{equation}
		Then  combining \eqref{lower size}, \eqref{length beta} and \eqref{small ball},   we conclude
		$$\ell(\beta_z[z,\eta])\le \ell(\beta_z)\le C_1 R\le  J'd(\eta,K) \quad \text{ when } \eta\in\gamma_{z_R,y_R}. $$

On the other hand, when $\eta\in\gamma_x[x,x_R]$, since our construction yields $\beta_z[z,\,\eta]=\gz_z[z,\,\eta]$, which is particularly contained in a John curve, it follows that 
		$$ \ell(\beta_z[z,\eta])\le J d(\eta,\,K)\le J'd(\eta,K)\quad \text{ when }   \eta\in \gamma_z[z,z_R].$$
		This   implies \eqref{base point ball}.
		Moreover by Proposition~\ref{length car pro}, 
		every point $w\in car(\beta_z,J')$ also can be joined to $y_R$ by a rectifiable curve $\hat{\gamma}_w$ satisfying 	$$\ell(\hat{\gamma}_w)\le \ell(\beta_z)\quad \text{and }\quad car(\hat{\gamma}_w,J')\subset car(\beta_z,J').$$
		Hence, by employing  \eqref{base point ball}, the arbitrariness of $z$ gives the second formula in \eqref{close omega}. The first formula in \eqref{close omega} holds due to $z\in Cl(car(\beta_z,J'))$, the closure of the carrot, for any $z\in V_{j,R}$.
	\end{proof}
	
	We need two more technical lemmas.
	The first one states how to choose a smaller carrot in the union of two carrots. 
	
	\begin{figure}[ht]
		\centering
		\def\svgwidth{\columnwidth}
		\resizebox{0.7\textwidth}{!}{
\begingroup%
  \makeatletter%
  \providecommand\color[2][]{%
    \errmessage{(Inkscape) Color is used for the text in Inkscape, but the package 'color.sty' is not loaded}%
    \renewcommand\color[2][]{}%
  }%
  \providecommand\transparent[1]{%
    \errmessage{(Inkscape) Transparency is used (non-zero) for the text in Inkscape, but the package 'transparent.sty' is not loaded}%
    \renewcommand\transparent[1]{}%
  }%
  \providecommand\rotatebox[2]{#2}%
  \newcommand*\fsize{\dimexpr\f@size pt\relax}%
  \newcommand*\lineheight[1]{\fontsize{\fsize}{#1\fsize}\selectfont}%
  \ifx\svgwidth\undefined%
    \setlength{\unitlength}{209.76377106bp}%
    \ifx\svgscale\undefined%
      \relax%
    \else%
      \setlength{\unitlength}{\unitlength * \real{\svgscale}}%
    \fi%
  \else%
    \setlength{\unitlength}{\svgwidth}%
  \fi%
  \global\let\svgwidth\undefined%
  \global\let\svgscale\undefined%
  \makeatother%
  \begin{picture}(1,0.53631759)%
    \lineheight{1}%
    \setlength\tabcolsep{0pt}%
    \put(0,0){\includegraphics[width=\unitlength,page=1]{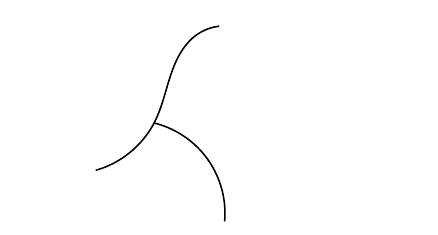}}%
    \put(0.325366,0.32447215){\makebox(0,0)[lt]{\lineheight{1.25}\smash{\begin{tabular}[t]{l}$\gamma_1$\end{tabular}}}}%
    \put(0.50592625,0.12335306){\makebox(0,0)[lt]{\lineheight{1.25}\smash{\begin{tabular}[t]{l}$\gamma_2$\end{tabular}}}}%
    \put(0.22078407,0.1296101){\makebox(0,0)[lt]{\lineheight{1.25}\smash{\begin{tabular}[t]{l}$z_1$\end{tabular}}}}%
    \put(0.52022805,0.02502817){\makebox(0,0)[lt]{\lineheight{1.25}\smash{\begin{tabular}[t]{l}$z_2$\end{tabular}}}}%
    \put(0.37363458,0.27173426){\makebox(0,0)[lt]{\lineheight{1.25}\smash{\begin{tabular}[t]{l}$y_2$\end{tabular}}}}%
    \put(0.51307715,0.4835797){\makebox(0,0)[lt]{\lineheight{1.25}\smash{\begin{tabular}[t]{l}$y_1$\end{tabular}}}}%
  \end{picture}%
\endgroup%
}
		\caption{The two curves $\gamma_1$ and $\gamma_2$ are presented, respectively, with their end points and the intersection point $y_2$. }
		\label{fig:gamma}
	\end{figure}
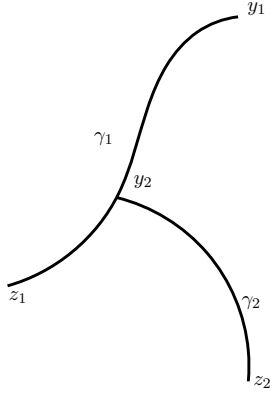
	
	\begin{lem}\label{Wj John}
		Let  $1\le J_1\le J_2$. 
		Assume that $z_1,z_2\in \mathbb R^n$ and $y_1\in\dot{\mathbb R}^n$. Let $\gamma_1$ be a rectifiable curve joining $z_1$ to $y_1$. If there exists a curve $\gamma_2$ joining $z_2$ to some point $y_2\in \gamma_1$, so that 
		\begin{equation}\label{joint}
			\frac{\ell(\gamma_2[z_2,y_2])}{J_2}\le \frac{\ell(\gamma_1[z_1,y_2])}{J_1},
		\end{equation}
		then, for any point $w\in \gamma_1[y_2,y_1)$, the curve $\hat{\gamma}:=\gamma_2\cup\gamma_1[y_2,w]$ joining $z_2$ to $w$ satisfies 
		
		\begin{equation}\label{inc 2}
			car(\hat \gamma,J_2)\subset car(\gamma_2,J_2)\cup car(\gamma_1,J_1).
		\end{equation}
		See Figure~\ref{fig:gamma} for a illustration. 
	\end{lem}
	
	\begin{proof}
		We first note that
        \begin{equation}\label{car include1}
            car(\hat{\gamma}[z_2,y_2],J_2)\subset car(\gamma_2,J_2).
        \end{equation}
		In addition, for any $a\in \gamma_1[y_2,w]$, the assumption
		$J_1\le J_2$ together with \eqref{joint} yields
		\begin{align}\nonumber
			\frac{\ell(\hat \gamma[z_2,a])}{J_2}&=\frac{\ell(\gamma_2[z_2,\,y_2])}{J_2} +\frac{\ell(\gamma_1[y_2,a])}{J_2}\nonumber\\
			&\le \frac{\ell(\gamma_1[z_1,y_2])}{J_1}+\frac{\ell(\gamma_1[y_2,a])}{J_1} 
			\le \frac{\ell(\gamma_1[z_1,a])}{J_1}.\nonumber 
		\end{align}
		As a result, for any $a\in \gamma_1[y_2,w]$, the definition of $car(\gamma_1, J_1)$ tells that 
        $$B\left(a, \frac{\ell(\hat \gamma[z_2,a])}{J_2}\right)\subset B\left(a,\frac{\ell(\gamma_1[z_1,a])}{J_1}\right)\subset car(\gamma_1, J_1).$$  
     Thus, by recalling the definition of $car(\hat \gamma,J_2)$ and \eqref{car include1}, we finally get 
		\eqref{inc 2}.
	\end{proof}

	\begin{lem}\label{cigar repla}
		Let $x,y,z\in \mathbb{R}^n$ and $J\ge 1$. Assume that there exist two curves $\gamma_{x,z},\gamma_{y,z}$ respectively joining $x,y$ to $z$. We denote the parametrization of
		$$\gamma:=\gamma_{x,z}\cup\gamma_{y,z}$$
		starting from $x$ and ending at $y$ as  $\gamma_1$, and the one in the reversed direction, starting from $y$ and ending at $x$, as 
		$\gamma_2$. Then there exists a ball $B$ with center $a\in \gamma$ satisfying 
		$$car(\gamma_1[x,a],J)\cup car(\gamma_2[y,a],J)\subset car(\gamma_{x,z},J)\cup car(\gamma_{y,z},J)$$
		and 
		 radius $r $ satisfying
         $$r=\frac{\ell(\gamma_1[x,a])}{J}=\frac{\ell(\gamma_2[y,a])}{J}.$$
	\end{lem}
	\begin{rem}
		Lemma \ref{cigar repla} is a corollary following from \cite[Theorem 3.6]{V1994} and \cite[Lemma 4.3]{V1994}. Since \cite[Lemma 4.3]{V1994} has used the concept of cigar in the statement, for the sake of completeness,  we   provide a  proof  avoiding the concept of  ``cigar'' in the Appendix~\ref{carrot to cigar}. 
	\end{rem}
	
	Now we are ready to prove Theorem~\ref{John component}.
	\begin{proof}[Proof of Theorem~\ref{John component}]

		We construct a sequence $\{W_{j,\infty}\}_{j\in\{1,\cdots,N\}}$  inductively.

		\noindent{\bf Step 1: Construct  $W_{1,\infty}$.}
		We start from a point $x_1\in \mathbb{R}^n\setminus K$ close to the origin. Then for any $R\ge 1$, the corresponding point $(x_1)_R \in\gamma_x\cap \partial B_{3R}$ is covered by some ball, say $D_R\in \mathcal{B}_{1,R}.$  Thus, by \eqref{VR}, we know that 
		
		\begin{equation}\label{x_1 belong}
			x_1\in V_{1,R}.
		\end{equation}
		Recall the definition \eqref{Omega def}. Let 
		
		$$W_{1,R}:=\Omega_{1,R,x_1},\quad \text{ and } \quad W_{1,\infty}:=\bigcup_{R\ge 1}W_{1,R}.$$
		Then from \eqref{x_1 belong} and \eqref{Omega def}, it follows that 
		$$car(\gamma_{x_1}[x_1,\,(x_{1})_R],\,J)\subset W_{1,\,R},\quad \text{ and } \quad car(\gamma_{x_1},J)\subset W_{1,\infty},$$
		and from Proposition \ref{omega cons} that $W_{1,\,R}$ is $J'$-carrot John domain with  $W_{1,R}\subset B_{2C_1R}$.

		\noindent{\bf Step 2: Proceeding inductively to construct  $\{W_{j,\infty}\}$.}	
		We run the induction based on the two subindices $j$ and $r$ for $W_{j,\,r}$. 
		
		For any $r>0$, define $x_{1,\,r}:=x_1$. Suppose that for some $m\ge 1$,  via the induction process, we have obtained points $\{x_j\}_{j=1}^m$ and the corresponding sets $\{W_{j,R}\}_{j=1}^{m}$ so that for any $ 1\le j\le m $, $R> |x_j|$ and some $r=r(R,\,j)>0$, 
		\begin{equation*}\label{introduction}
			W_{j,R}:=\Omega_{j,R,x_{j,r}},\quad  \text{ and } \quad W_{j,\infty}:=\bigcup_{R> |x_j|} W_{j,R}\quad\text{ for some } r<R. 
			\qquad  
		\end{equation*}

		Suppose that,  for  some $s>0$
		\begin{equation}\label{uncover}
			B_{s}\setminus \left(K \cup \bigcup_{j=1}^m  \overline{W}_{j,\,s} \right)\neq \emptyset.
		\end{equation}
		This yields the existence of another point in $B_{s}\setminus \left(K \cup \bigcup_{j=1}^m  \overline{W}_{j,\,s} \right)$. Take $r>0$ to be (almost) the smallest $s>0$ for which \eqref{uncover} holds. Next, we consider two cases.
		
		\noindent{\bf Case 1: } Suppose that
		\begin{equation}\label{case 1}
			B_r\setminus K\subset \bigcup_{R\ge r}\bigcup_{j=1}^{m}V_{j,R}.
		\end{equation}

		\begin{figure}[ht]
			\centering
			\def\svgwidth{\columnwidth}
			\resizebox{0.7\textwidth}{!}{
\begingroup%
  \makeatletter%
  \providecommand\color[2][]{%
    \errmessage{(Inkscape) Color is used for the text in Inkscape, but the package 'color.sty' is not loaded}%
    \renewcommand\color[2][]{}%
  }%
  \providecommand\transparent[1]{%
    \errmessage{(Inkscape) Transparency is used (non-zero) for the text in Inkscape, but the package 'transparent.sty' is not loaded}%
    \renewcommand\transparent[1]{}%
  }%
  \providecommand\rotatebox[2]{#2}%
  \newcommand*\fsize{\dimexpr\f@size pt\relax}%
  \newcommand*\lineheight[1]{\fontsize{\fsize}{#1\fsize}\selectfont}%
  \ifx\svgwidth\undefined%
    \setlength{\unitlength}{209.76377106bp}%
    \ifx\svgscale\undefined%
      \relax%
    \else%
      \setlength{\unitlength}{\unitlength * \real{\svgscale}}%
    \fi%
  \else%
    \setlength{\unitlength}{\svgwidth}%
  \fi%
  \global\let\svgwidth\undefined%
  \global\let\svgscale\undefined%
  \makeatother%
  \begin{picture}(1,0.53631759)%
    \lineheight{1}%
    \setlength\tabcolsep{0pt}%
    \put(0,0){\includegraphics[width=\unitlength,page=1]{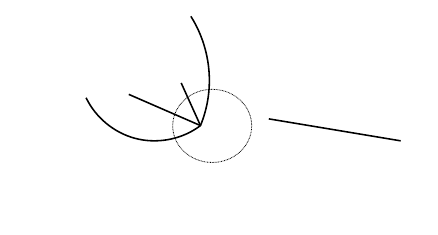}}%
    \put(0.56660907,0.18321345){\makebox(0,0)[lt]{\lineheight{1.25}\smash{\begin{tabular}[t]{l}$R$\end{tabular}}}}%
    \put(0.69362038,0.09876402){\makebox(0,0)[lt]{\lineheight{1.25}\smash{\begin{tabular}[t]{l}$R''$\end{tabular}}}}%
    \put(0.33879827,0.32766992){\makebox(0,0)[lt]{\lineheight{1.25}\smash{\begin{tabular}[t]{l}$V_{2,\,R'}$\end{tabular}}}}%
    \put(0.41909867,0.33854472){\makebox(0,0)[lt]{\lineheight{1.25}\smash{\begin{tabular}[t]{l}$V_{3,\,R}$\end{tabular}}}}%
    \put(0.26306752,0.42667909){\makebox(0,0)[lt]{\lineheight{1.25}\smash{\begin{tabular}[t]{l}$V_{1,\,R''}$\end{tabular}}}}%
    \put(0,0){\includegraphics[width=\unitlength,page=2]{tree.pdf}}%
    \put(0.63088523,0.12659942){\makebox(0,0)[lt]{\lineheight{1.25}\smash{\begin{tabular}[t]{l}$R'$\end{tabular}}}}%
    \put(0.73369747,0.27719956){\makebox(0,0)[lt]{\lineheight{1.25}\smash{\begin{tabular}[t]{l}$V_{2,\,R''}$\end{tabular}}}}%
    \put(0.61756611,0.03622702){\makebox(0,0)[lt]{\lineheight{1.25}\smash{\begin{tabular}[t]{l}$V_{3,\,R''}$\end{tabular}}}}%
  \end{picture}%
\endgroup%
}
			\caption{The set $W_{3,\,R}$ is contained in $W_{2,\,R'}$, and $W_{2,\,R'}$ is contained in $W_{1,\,R''}$, where  $R\le R'\le  R''$. Eventually they are all contained in $W_{1,\,\infty}$. However, we note that $W_{3,\,R}$ and $W_{3,\,R''}$ could have no intersection. }
			\label{fig:tree}
		\end{figure}

		Since $V_{j,\,r}$ is a decomposition of $B_r\setminus K$ for any $r>0$,  \eqref{uncover} and \eqref{close omega} imply that there exists some  point $x_{m+1,\,r}\in V_{{m+1},r}$,
		and
		\begin{equation}\label{gamma m+1}
			car\left(\gamma_{x_{m+1,\,r}}[x_{m+1,\,r},\,(x_{m+1,\,r})_r],J\right)\subset  \Omega_{m+1,r,\,x_{m+1,\,r}} 
		\end{equation}
		according to Proposition \ref{omega cons}.
		Now
		from \eqref{case 1} it follows that for some $R> r$ , we have
		\begin{equation}\label{be eaten}
			x_{m+1,r}\in V_{k,\,R}\neq \emptyset, \quad \text{ for some $1\le k\le m$.}
		\end{equation}
		Let $R'$ be the infimum among all positive number for which \eqref{be eaten} happens. 
		Then if $R'>r$, we define
		$$W_{m+1,\,s}:=\Omega_{{m+1},\,s,\,x_{m+1,\,r}}.$$
		for $r\le s< R'$. If $R'=r$, then we only define $W_{m+1,\,s}:=\Omega_{{m+1},\,s,\,x_{m+1,\,r}}$ for $s=r$.

		In addition, it follows from \eqref{gamma m+1} that
		$$car\left(\gamma_{x_{m+1,\,r}}[x_{m+1,\,r},\,(x_{m+1,\,r})_{s}],J\right)\subset W_{m+1,\,s}$$
		and from Proposition \ref{omega cons} that $W_{m+1,\,s}$ is  $J'$-carrot John domain with  $W_{m+1,s}\subset B_{2C_1s}$.
		
		Next we check if
		\begin{equation}\label{nonempty}
			B_{r}\setminus \left(K \cup \bigcup_{j=1}^{m+1}  \overline{W}_{j,\,r} \right)\neq \emptyset.
		\end{equation}
		If it is non-empty, we continue to define $W_{m+2,\,r}$ and iterate our process. Otherwise, we define $W_{j,r}:=\emptyset$ for $j\in\{m+1,\cdots, \hat N\}$. Then increase $r$ until  \eqref{nonempty} holds for some $r'\ge r$, and consider the set $W_{m+1,\,r'}$.

		\noindent{\bf Case 2: }	If \eqref{case 1} fails, then there exists a point $x_{m+1,r}\notin V_{j,R}$ for any $k\in\{1,\cdots,m\}$ and $R\ge r$. Since $\{V_{j,\,R}\}$ decomposes 
		$B_R\setminus K$, then for every $R>r$, there exists $k_R\in\{m+1,\cdots,\hat N\}$ such that
  $x_{m+1,r}\in V_{k_R,R}$. Then up to relabeling the first subindex of $\{V_{j,R}\}_{j=m+1}^{\hat N}$, we may assume that  $x_{m+1,r}\in V_{m+1,R}$. 
  
As we need to define $W_{m+1,\,\infty}$ later, in order to distinguish from the first case, we write $x_{m+1}:=x_{m+1,r}$ (also recall that $x_1:=x_{1,\,r}$ at the beginning of Step 2). Then define 
		$$W_{m+1,\,s}:=\Omega_{m+1,\,s,\,x_{m+1}} \quad \text{  for all $s> |x_{m+1}|$}, $$
		and let 
		\begin{equation}\label{infty constr}
			W_{m+1,\infty}:=\bigcup_{R> |x_{m+1}|}W_{m+1,R}. 
		\end{equation}
		Likewise,   
		\eqref{close omega} gives
		$$ car(\gamma_{x_{m+1}}[x_{m+1},\,(x_{m+1})_R],\,J)\subset W_{{m+1},\,R}\quad\forall R>|x_{m+1}|,\quad car(\gamma_{x_{m+1}},J)\subset W_{m+1,\infty},$$
		and from Proposition \ref{omega cons} that $W_{m+1,\,R}$ is  $J'$-carrot John domain with  $W_{m+1,R}\subset B_{2C_1R}$. 
		
		\noindent{\bf Step 3: Uniformly finitely many $W_{j,\,\infty}$.}	
		
		Our process is stopped when, for any $R>0$, 
		
		$$			B_{R}\setminus \left(K \cup \bigcup_{j=1}^{\hat{N}}  \overline{W}_{j,\,R} \right)= \emptyset$$
		and in particular, all $W_{j,\infty}$ have been founded so that 
		\begin{equation}\label{infty cover}
			\mathbb{R}^n=\left(K \cup \bigcup_{j=1}^N \overline{W}_{j,\infty} \right)
		\end{equation}
		for some $N\le \hat{N}$. Suppose that \eqref{infty cover} is not true. Then there exists a point
		$$z\in \mathbb{R}^n\setminus  \left(K \cup \bigcup_{j=1}^N \overline{W}_{j,\infty}\right).$$
		Further observe that  \eqref{infty constr} and \eqref{close omega} give $V_{j,R}\subset \overline{W}_{j,\infty}$.  Then our induction process tells that we can obtain  a new set $W_{N+1,\infty}$ according to $z\notin V_{j,R}$ for any $j\in\{1,\cdots,N\}$ and those sufficiently large  $R$, which is impossible.

		Moreover,  for any $1\le j\le \hat N$, $W_{j,\,R}$ is a $J'$-carrot John domain with  $W_{j,R}\subset B_{2C_1R}$, and for each $R\ge 1$, there exists $0<r\le R$ so that 
		\begin{equation}\label{carrot in W}
			car(\gamma_{x_j,\,r}[x_{j,\,r},\,(x_{j,\,r})_R],\,J)\subset W_{j,\,R}
		\end{equation}
		and
		\begin{equation}\label{dage}
			car(\gamma_{x_{j}},J)\subset W_{j,\infty}\quad \text{for any }1\le j\le N. 
		\end{equation}

		\noindent{\bf Step 4:  $W_{j,\infty}$ is $J'$-carrot John  with John center $\infty$.}
		Fix $j\in\{1,\cdots,N\}$. 
		For every $z\in W_{j,\infty}$, it follows that	$z\in W_{j,\,R}$ for some $R>|x_j|$. Hence, thanks to Proposition \ref{omega cons}, $z$ can be joined to  $(x_j)_{R}$ by a rectifiable curve $\beta_z$ as the core of a $J'$-carrot satisfying  
		\begin{equation}\label{relate4}
			\ell(\beta_z)\le C_1 R \quad   \text{ and }\quad car(\beta_z,J')\subset W_{j,\infty},
		\end{equation}
		In addition, by employing the definition of $J'$ \eqref{carrot construct 2} and \eqref{lower size}, Proposition \ref{omega cons} tells
	\begin{equation}\label{relate5}
			\frac{\ell(\beta_z)}{J'}\le \frac{R}{J}\le \frac{\ell(\gamma_{x_j}[x_j,(x_j)_{R}])}{J}. 
		\end{equation}
		Further note that  $(x_j)_{R}\in \gamma_{x_j}$. Then by employing \eqref{relate5},  Lemma~\ref{Wj John} tells that the curve $\zeta_z:=\beta_z\cup \gamma_{x_j}[(x_j)_{R},\infty)$ joining $z$ toward $\infty$ satisfies  
		\begin{equation*}
			car (\zeta_z,J')\subset car(\beta_z,J')\cup car(\gamma_{x_{j}},J).
		\end{equation*}		
		Moreover, \eqref{dage} and \eqref{relate4} yield
		$$		car(\beta_z,J')\cup car(\gamma_{x_{j}},J) \subset W_{j,\infty}. $$
		Thus the arbitrariness of $z$ implies that $W_{j,\,\infty}$ is $J'$-carrot John with John center $\infty$. 
		
		\noindent{\bf Step 5:   Proof of \eqref{Boman chain} and \eqref{cigar ball}.} In addition, for each pair of  points $z,w\in W_{j,\infty}$,  we can find $R_z,\,R_w>|x_j|$, such that $z\in W_{j,R_z}$ and $w\in W_{j,R_w}$. We may assume  $R_z\le R_w$. Then  Step 4 gives us  two curves $\bz_z,\,\bz_w$ joining $z,w$ to $(x_j)_{R_z},(x_j)_{R_w}$, respectively, such that 
		$$car(\beta_w,J')\subset W_{j,\infty},\quad car(\beta_z,J')\subset W_{j,\infty}\quad\text{and}\quad  \frac{\ell(\beta_z)}{J'}\le \frac{\ell(\gamma_{x_j}[x_j,(x_j)_{R_z}])}{J};$$
		see \eqref{relate4} and \eqref{relate5}. 
		Therefore, applying \eqref{dage} and Lemma \ref{Wj John} with $\gamma_1=\gamma_{x_j}[x_j,(x_j)_{R_w}]$,  $ \gamma_2= \beta_z$, and
		$$J=J_1\le J_2=J',$$
		there is a curve $\hat{\gamma}:=\beta_z\cup \gamma_{x_j}[(x_j)_{R_z},(x_j)_{R_w}]$ joining $z$ to $(x_j)_{R_w}$, such that
		$$car(\hat{\gamma},J')\subset car(\beta_z,J')\cup car(\gamma_{x_j},J)\subset W_{j,\infty}.$$
		Then, by Lemma \ref{cigar repla}, we finally arrive at \eqref{Boman chain} and \eqref{cigar ball}.

		\noindent{\bf Step 6: Proof of \eqref{volumn relat} and \eqref{volumn relat 2}.}		The remaining task is to prove \eqref{volumn relat} and \eqref{volumn relat 2}. As $W_{j,R}$ is a $J'$-carrot John domain with center $(x_{j,r})_{R}$ with some $r<R$, it follows from the definition of John  domain that
		$$B\left(x_{j,r},\frac{\ell(\gamma_{x_{j,r}}[x_{j,r},(x_{j,r})_{R}])}{J}\right)\subset car(\gamma_{x_{j,r}}[x_{j,r},(x_{j,r})_{R}],J)\subset W_{j,R}\subset B\left((x_{j,r})_{R},J'd((x_{j,r})_{R},K)\right).$$
		As a result, by \eqref{lower size} and \eqref{radius larg}, the above estimate yields that 
		$$C(n,\,J)^{-1}R^n\le |W_{j,R}|\le C(n,\,J)R^n$$
		and then the inequality  \eqref{volumn relat} follows.

		Furthermore,  given $k\in \{1,\cdots, \hat N\}$, we consider the set $W_{k,\,R}$ which  contains the carrot
		$$car(\gamma_{x_{k,\,r}}[x_{k,\,r},\,(x_{k,\,r})_R],\,J)$$
		by \eqref{carrot in W}. Then  we choose $1\le k_l\le \hat N$ so that
		$x_{k,\,r}\in V_{k_l,\,2^l R};$
		such a $k_l$ exists since $\{V_{j,\,2^l R}\}_j$ covers $B_{2^l R}\setminus K. $

		Toward the inequality \eqref{volumn relat 2}, recall that  $W_{k_l,\,2^l R}$ is constructed via Proposition~\ref{omega cons}, which, in particular by the definition of $\beta_{x_{k,\,r}}$, contains the carrot
		$$car(\gamma_{x_{k,\,r}}[x_{k,\,r},\,(x_{k,\,r})_{2^l R}],\,J'); $$
		recall that
		$$\gamma_{x_{k,\,r}}[x_{k,\,r},\,(x_{k,\,r})_{2^l R}]\subset \beta_{x_{k,\,r}}.$$
		Especially, 
		$$car(\gamma_{x_{k,\,r}}[x_{k,\,r},\,(x_{k,\,r})_{2^l R}],\,J')\subset W_{k_l,\,2^l R}\cap W_{k_{l+1},\,2^{l+1} R}, $$
		and  \eqref{volumn relat 2} follows from \eqref{volumn relat} as
		$$\Big|car(\gamma_{x_{k,\,r}}[x_{k,\,r},\,(x_{k,\,r})_{2^l R}],\,J')\Big|\ge C(n,\,J) (2^l R)^{n}.$$ 
		
	\end{proof}

	\appendix

	\section{Proof of the estimate \eqref{lower es}}\label{lower bound}
	\begin{proof}
		Let  $$\hat{J}:=J(\hat{x},\Omega;\hat{y}).$$
		Then using Lemma \ref{J eqiup}, there exists a rectifiable curve  $\beta\subset\Omega$ joining $\hat{x}$ to $\hat{y}$ together with the corresponding $\hat{J}$-carrot $car(\beta,\hat{J})$, such that 
		\begin{equation}\label{looking for F}
			\sup_{t\in [0,1]}j(t;\hat{x},\beta,\Omega)=J(\hat{x},\Omega;\hat{y})=\hat{J}\quad \text{and}\quad car(\beta,\hat{J})\subset\Omega.
		\end{equation}
		Analogously, thanks to the compactness of $[0,1]$, the definition of $j(t;\hat{x},\beta,\Omega)$ tells that we can find a point $\hat s\in [0,1]$, such that 
		$$\frac{\ell_{\|\cdot\|}(\beta([0,\hat s]))}{d_{\Vert\cdot\Vert}(\beta(\hat s),\partial\Omega)}=\hat{J}=\sup_{t\in 
 [0,1]}j(t;\hat{x},\beta,\Omega).$$

		We repeat the argument by replacing $\gamma,x$ and $J$ respectively by $\beta,\hat{x}$ and $\hat{J}$ in the proof of \eqref{sp1.5}. Then we have 
		\begin{equation}\label{replace 1.5}
			d_{\Vert \cdot\Vert}(\beta,\partial\Omega)\ge\frac{d_{\Vert \cdot\Vert}(\hat{x},\partial\Omega)}{2C_{\|\cdot\|}\hat{J}}.
		\end{equation}
		Further assume that $L_{x,\hat{x}}\subset\Omega$ is the line segment joining  
		$x$ to $\hat{x}$ and $L_{\hat{y},y}\subset\Omega$ is the one joining
		$\hat{y}$ to $y$.  Then 
		$$\hat{\beta}:=L_{x,\hat{x}}\cup \beta\cup L_{\hat{y},y}$$
		is a rectifiable curve within $\Omega$ joining $x$ to $y$. Now we also repeat the argument by replacing $s,J,\gamma$ and $\hat{\gamma}$ with $\hat s,\hat{J},\beta$ and $\hat{\beta}$, respectively, and swapping $x,y$ respectively with $\hat{x},\hat{y}$, respectively. By letting $(\hat x,\hat y)$ close enough to $(x,y)$, \eqref{sp3} changes into 
		\begin{multline}\label{estimate on hat f}
			\frac{\ell_{\|\cdot\|}(\hat \beta([0,t]))}{d_{\Vert\cdot\Vert}(\hat \beta(t),\partial\Omega)}
			-\frac{\ell_{\|\cdot\|}(\beta([0,\hat s]))}{d_{\Vert\cdot\Vert}(\beta(\hat s),\partial\Omega)}\\
			\le\max \left\{
			\frac{\Vert \hat{x}-x\Vert}{d_{\Vert\cdot\Vert}(\beta,\partial\Omega)},\,
			\frac{C(n,C_{\|\cdot\|},\, \hat{J}) }{ d_{\Vert\cdot\Vert}(\hat{y},\partial\Omega)}\left(\Vert \hat{x}-x\Vert +\Vert \hat{y}-y\Vert \right) ,\,
			\frac{2\Vert \hat{x}-x\Vert}{d_{\Vert\cdot\Vert}(\hat{x},\partial\Omega) } \right\}
		\end{multline}
		for any $ z\in \hat{\beta}$.
		Further note that when $\Vert x-\hat{x}\Vert+\Vert y-\hat{y}\Vert<\delta$ for a sufficiently small and  positive constant $\delta=\delta(x,y,C_{\|\cdot\|})$ satisfying $\delta\le\frac{1}{2}\min\{d_{\Vert\cdot\Vert}(
		x,\partial\Omega),d_{\Vert\cdot\Vert}(
		y,\partial\Omega)\}$ at least, by \eqref{sp1} and \eqref{looking for F},
		the estimate \eqref{upper}
		gives  
		\begin{equation}\label{hat J control}
			\hat{J}\le C(n,C_{\|\cdot\|},\,J).
		\end{equation}
		Consequently, combining the construction of $\hat{\beta}$, \eqref{estimate on hat f}, \eqref{replace 1.5} and \eqref{hat J control}, it follows that  when $\Vert x-\hat{x}\Vert+\Vert y-\hat{y}\Vert<\delta$,
		\begin{align}\label{low estimate}
			&J(x,\Omega;y)-J(\hat{x},\Omega;\hat{y})\nonumber\\
			\le&\sup_{t\in [0,1]}\frac{\ell_{\|\cdot\|}(\hat \beta([0,t]))}{d_{\Vert\cdot\Vert}(\hat \beta(t),\partial\Omega)}
			-\frac{\ell_{\|\cdot\|}(\beta([0,\hat s]))}{d_{\Vert\cdot\Vert}(\beta(\hat s),\partial\Omega)}\nonumber\\
			\le &\max \left\{
			\frac{\Vert \hat{x}-x\Vert}{d_{\Vert\cdot\Vert}(\beta,\partial\Omega)},\,
			\frac{C(n,C_{\|\cdot\|},\, \hat{J}) }{ d_{\Vert\cdot\Vert}(\hat{y},\partial\Omega)}\left(\Vert \hat{x}-x\Vert +\Vert \hat{y}-y\Vert \right) ,\,
			\frac{2\Vert \hat{x}-x\Vert}{d_{\Vert\cdot\Vert}(\hat{x},\partial\Omega) } \right\}\nonumber\\
			\le &  \frac{C(n,C_{\|\cdot\|},\, \hat{J}) }{d_{\Vert\cdot\Vert}(\beta,\partial\Omega)}\left(\Vert \hat{x}-x\Vert +\Vert \hat{y}-y\Vert \right)\le \frac{C(n,C_{\|\cdot\|},\, \hat{J}) }{d_{\Vert\cdot\Vert}(\hat{x},\partial\Omega)}\left(\Vert \hat{x}-x\Vert +\Vert \hat{y}-y\Vert \right)\nonumber \\
			\le &\frac{C(n,C_{\|\cdot\|},\, \hat{J}) }{d_{\Vert\cdot\Vert}(x,\partial\Omega)}\left(\Vert \hat{x}-x\Vert +\Vert \hat{y}-y\Vert \right)\le\frac{C(n,C_{\|\cdot\|},\, J) }{d_{\Vert\cdot\Vert}(x,\partial\Omega)}\left(\Vert x-\hat{x}\Vert +\Vert y-\hat{y}\Vert \right),
		\end{align}
		which yields \eqref{lower es}.
		
	\end{proof}
	
	

	\section{Proof of Lemma \ref{cigar repla}}\label{carrot to cigar}
	\begin{proof}
		We may assume that $\ell(\gamma_{x,z})\ge \ell(\gamma_{y,z})$.  Then there exists a point $a \in \gamma_{x,z}$, such that 
        \begin{equation}\label{length form}
            \ell(\gamma_{x,z}[x,a])=\ell(\gamma_{y,z})+\ell(\gamma_{x,z}[z,a]).
        \end{equation}
Note that the construction of $\gamma $ tells that 
\begin{equation}\label{curve combine}
    \gamma_1[x,a]=\gamma_{x,z}[x,a],\quad \gamma_2[y,a]=\gamma_{y,z}\cup\gamma_1[z,a]. 
\end{equation}
Then, due to \eqref{length form} and \eqref{curve combine}, it follows that  $\ell(\gamma_1[x,a])=\ell(\gamma_2[y,a])$. This implies that 
 for  each $\eta\in \gamma_2[z,a]=\gamma_{x,z}[z,a]$, 
\begin{equation}\label{length compare}
\ell(\gamma_2[y,\eta])\le \ell(\gamma_2[y,a])=\ell(\gamma_1[x,a])\le \ell(\gamma_1[x,\eta]).    
\end{equation}	
Besides, \eqref{curve combine} directly yields that
	\begin{equation}\label{sp1212}
			car(\gamma_1[x,a],J)\subset car(\gamma_{x,z},J),\quad  car(\gamma_{y,z},J)= car(\gamma_2[y,z],J),
		\end{equation}
		which, together with the definition of $car(\gamma_{y,z},J)$ and \eqref{length compare}, implies that 	$$car(\gamma_1[x,a],J)\cup car(\gamma_2[y,a],J)\subset car(\gamma_{x,z},J)\cup car(\gamma_{y,z},J).$$
		As a result, the desired  ball is   $$B=B\left(a,\frac{\ell(\gamma_1[x,a])}{J}\right).$$ 
        The proof is completed.
	\end{proof}

\end{document}